\documentclass{amsart}

\usepackage[utf8]{inputenc} 
\usepackage[english]{babel} 
\usepackage{microtype} 
\usepackage{mathrsfs} 
\usepackage{url} 
\usepackage{amssymb}

\usepackage{pgfplots} \pgfplotsset{compat=1.17} \usepgfplotslibrary{fillbetween} \pgfplotsset{every tick label/.append style={font=\tiny}}

\numberwithin{equation}{section}

\theoremstyle{plain} 
\newtheorem{theorem}{Theorem}[section] 
\newtheorem{lemma}[theorem]{Lemma} 
 
\newtheorem*{theoremA}{Theorem A} 
\newtheorem*{theoremB}{Theorem B} 
\newtheorem*{theoremC}{Theorem C} 
\newtheorem*{theoremD}{Theorem D}

\theoremstyle{remark} 
\newtheorem*{remark}{Remark}

\theoremstyle{definition} 
\newtheorem*{definition}{Definition}

\DeclareMathOperator{\mre}{Re} 
 
\DeclareMathOperator{\Dom}{Dom} 
\DeclareMathOperator{\ac}{ac}

\begin{document} 
\title{The spectrum of some Hardy kernel matrices} 
\date{\today} 

\author{Ole Fredrik Brevig} 
\address{Department of Mathematics, University of Oslo, 0851 Oslo, Norway} 
\email{obrevig@math.uio.no}

\author{Karl-Mikael Perfekt} 
\address{Department of Mathematical Sciences, Norwegian University of Science and Technology (NTNU), 7491 Trondheim, Norway} 
\email{karl-mikael.perfekt@ntnu.no}

\author{Alexander Pushnitski} 
\address{Department of Mathematics, King’s College London, Strand, London WC2R 2LS, United Kingdom} 
\email{alexander.pushnitski@kcl.ac.uk} 
\begin{abstract}
	For $\alpha > 0$ we consider the operator $K_\alpha \colon \ell^2 \to \ell^2$ corresponding to the matrix
	\[\left(\frac{(nm)^{-\frac{1}{2}+\alpha}}{[\max(n,m)]^{2\alpha}}\right)_{n,m=1}^\infty.\]
	By interpreting $K_\alpha$ as the inverse of an unbounded Jacobi matrix, we show that the absolutely continuous spectrum coincides with $[0, 2/\alpha]$ (multiplicity one), and that there is no singular continuous spectrum. There is a finite number of eigenvalues above the continuous spectrum. We apply our results to demonstrate that the reproducing kernel thesis does not hold for composition operators on the Hardy space of Dirichlet series $\mathscr{H}^2$. 
\end{abstract}

\subjclass[2020]{Primary 47A10. Secondary 47B33, 47B34, 47B36.}

\maketitle

\section{Introduction}

\subsection{Hardy kernels} A \emph{Hardy kernel} is a real-valued function $k=k(x,y)$ of two variables $x>0$ and $y>0$ which is: 
\begin{itemize}
	\item symmetric: $k(x,y)=k(y,x)$; 
	\item homogeneous of degree $-1$: $k(\lambda x,\lambda y)=\lambda^{-1}k(x,y)$; 
	\item satisfies the condition
	\[\int_0^\infty \frac{|k(x,1)|}{\sqrt{x}}\,dx<\infty.\]
\end{itemize}
Given a Hardy kernel, one can associate with it an integral operator $\mathbf{K}$ in $L^2(\mathbb{R}_+)$
\[\mathbf{K} u(x)=\int_0^\infty k(x,y)u(y)\,dy, \qquad u\in L^2(\mathbb{R}_+),\]
and an operator on $\ell^2=\ell^2(\mathbb{N})$, given by
\[K u(n)=\sum_{m=1}^\infty k(n,m) u(m), \qquad u\in \ell^2.\]
In other words, $K$ is the ``infinite matrix''
\[K=\left(k(n,m)\right)_{n,m=1}^\infty.\]
The study of the boundedness of both the continuous version $\mathbf{K}$ and the discrete version $K$ is implicit in the work of Schur \cite{Schur}; a systematic account can be found in Hardy, Littlewood and Polya \cite[Chapter IX]{HLP}. 

Due to the homogeneity condition $\mathbf{K}$ commutes with the unitary group of dilations in $L^2(\mathbb{R}_+)$ and is therefore diagonalised by the Mellin transform. More precisely, the unitary Mellin transform $\mathscr{M} \colon L^2(\mathbb{R}_+)\to L^2(\mathbb{R})$, defined by
\[\mathscr{M} f(t)=\frac{1}{\sqrt{2\pi}}\int_0^\infty f(x)\,x^{-\frac{1}{2}-it}\,dx,\]
transforms $\mathbf{K}$ into the operator of multiplication by the function
\[\omega(t)=\int_0^\infty k(x,1)\,x^{-\frac{1}{2}-it}dx\]
in $L^2(\mathbb{R})$. From here one can read off the spectral properties of $\mathbf{K}$. 

It is by no means clear how to diagonalise $K$ or how to relate the spectral properties of $K$ to those of $\mathbf{K}$. In general, there is no simple answer to these questions, as no discrete analogue of the Mellin transform is available. The main example known to us when these questions have been answered is the Carleman kernel:
\[k_{\operatorname{c}}(x,y)=\frac{1}{x+y},\]
because in this case the corresponding discrete version
\[K_{\operatorname{c}}=\left(\frac{1}{n+m}\right)_{n,m=1}^\infty\]
is a variant of Hilbert's matrix, diagonalised by Rosenblum in \cite{Rosenblum58}. Incidentally, in this case the spectrum of both discrete and continuous operator is purely absolutely continuous (a.c.), with 
\begin{align}
	\sigma_{\ac}(\mathbf{K}_C)=[0,\pi]\quad \text{with multiplicity \underline{two},} \label{eq:Carleman1} \\
	\sigma_{\ac}(K_C)=[0,\pi]\quad \text{with multiplicity \underline{one}.} 
\label{eq:Carleman2} \end{align}

\subsection{The kernels $k_\alpha$} The purpose of this paper is to exhibit a special family of Hardy kernels for which the spectral analysis of both $\mathbf{K}$ and $K$ can be performed. For $0<\alpha<\infty$ we consider
\[k_\alpha(x,y)=\frac{(xy)^{-\frac{1}{2}+\alpha}}{[\max(x,y)]^{2\alpha}}, \qquad x,y>0.\]
As it follows from the previous discussion, the corresponding integral operator $\mathbf{K}_\alpha$ in $L^2(\mathbb{R}_+)$ is unitarily equivalent to the operator of multiplication by the function 
\begin{equation}\label{eq:omega} 
	\omega_\alpha(t) = \int_0^\infty k_\alpha(x,1)\,x^{-\frac{1}{2}-it}\,dx = \frac{2\alpha}{\alpha^2+t^2}, \qquad t\in\mathbb{R} 
\end{equation}
in $L^2(\mathbb{R})$. From here we read off the spectral structure of $\mathbf{K}_\alpha$: it has a purely a.c. spectrum which coincides with the range of $\omega_\alpha$,
\[\sigma_{\ac}(\mathbf{K}_\alpha)=[0,\tfrac{2}{\alpha}]\quad\text{with multiplicity \underline{two}}.\]
Next, we consider the discrete version of $\mathbf{K}_\alpha$, i.e.
\[K_\alpha=\left(\frac{(nm)^{-\frac{1}{2}+\alpha}}{[\max(n,m)]^{2\alpha}}\right)_{n,m=1}^\infty.\]
It is not difficult to see that $K_\alpha$ is bounded on $\ell^2$ (see e.g. \cite[Section~9.2]{HLP}). The family of matrices $K_\alpha$ appeared in the work of the first-named author \cite{Brevig17}, where it was demonstrated that 
\begin{equation}\label{eq:Brevigest} 
	\max\left(\frac{2}{\alpha},\zeta(1+2\alpha)\right) \leq \|K_\alpha\| \leq \max\left(\frac{2}{\alpha},\zeta(1+\alpha)\right), 
\end{equation}
where $\zeta(s) = \sum_{n\geq1} n^{-s}$ denotes the Riemann zeta function. From this it is easy to see that $\|K_\alpha\|=2/\alpha$ for $0<\alpha\leq 1$ and $\|K_\alpha\|>2/\alpha$ for $2\leq\alpha<\infty$.

The quadratic form of $K_\alpha$ has the integral representation (here and in what follows $\langle \cdot,\,\cdot \rangle$ stands for the standard inner product in $\ell^2$) 
\begin{equation}\label{eq:intrep} 
	\langle K_\alpha x,\, x\rangle = \int_{-\infty}^\infty \left|\sum_{n=1}^\infty x(n) n^{-\frac{1}{2}-it}\right|^2 \frac{2\alpha}{\alpha^2+t^2}\,\frac{dt}{2\pi}. 
\end{equation}
This representation can be established either directly by expanding and computing the integral using the residue theorem, or from \eqref{eq:omega} through the Mellin transform.

Although we are not able to diagonalise $K_\alpha$ explicitly (except in the very special case $\alpha=1/2$ --- see Section~\ref{sec:jacint} below), it is possible to describe the spectral structure of $K_\alpha$. Here is our first main result: 
\begin{theoremA}
	For $0<\alpha<\infty$, the a.c. spectrum of $K_\alpha$ is
	\[\sigma_{\ac}(K_\alpha)=[0,\tfrac{2}{\alpha}]\quad\text{with multiplicity \underline{one}.}\]
	The singular continuous spectrum of $K_\alpha$ is empty. $K_\alpha$ has finitely many (possibly none) eigenvalues, all of which are simple and located in the interval $(\frac{2}{\alpha},\infty)$. 
\end{theoremA}

Observe that the a.c. spectra of $\mathbf{K}_\alpha$ and $K_\alpha$ coincide as sets, but with different multiplicities --- the same phenomenon as for the Carleman kernel \eqref{eq:Carleman1} and \eqref{eq:Carleman2}. 

We can give some qualitative statements about the behaviour of the eigenvalues of $K_\alpha$ as $\alpha$ increases (by Theorem A, they are all located in the interval $(\tfrac2{\alpha},\infty)$). 

Let us enumerate them in the decreasing order:
\[\lambda_1(K_\alpha)>\lambda_2(K_\alpha)>\cdots>\frac{2}{\alpha}.\]
We will denote by $N(K_\alpha)$ the total number of eigenvalues of $K_\alpha$. Note that by \eqref{eq:Brevigest}, we know that $N(K_\alpha)=0$ for $0<\alpha\leq1$ and $N(K_\alpha)>0$ for $2\leq \alpha<\infty$. We can also read off from \eqref{eq:Brevigest} that $\lambda_1(K_\alpha)=\|K_\alpha\|\to1$ as $\alpha\to\infty$. 
\begin{theoremB}
	\mbox{} 
	\begin{enumerate}
		\item[\rm (i)] For every fixed $j\geq1$, the function $\alpha\mapsto \lambda_j(\alpha K_\alpha)$ is increasing. 
		\item[\rm (ii)] The function $\alpha \mapsto N(K_\alpha)$ is non-decreasing and unbounded. 
		\item[\rm (iii)] For every fixed $j\geq1$, as $\alpha\to\infty$, it holds that
		\[\lambda_j(K_\alpha) =j^{-1}+O(\alpha^{-1}).\]
	\end{enumerate}
\end{theoremB}

\begin{figure}
	\centering 
	\begin{tikzpicture}[scale=1.5]
		\begin{axis}
			[axis equal image, 
			axis lines=middle, 
			axis line style=thin, 
			axis on top, 
			xmin=0, 
			xmax=12.5, 
			ymin=0, 
			ymax=6.5, 
			ytick={1,2,3,4,5,6}, 
			yticklabels={2,4,6,8,10,12}, 
			ylabel=$\scriptstyle\lambda$, 
			xlabel=$\scriptstyle\alpha$, 
			every axis x label/.style={ at={(ticklabel* cs:1.025)}, anchor=west,}, 
			every axis y label/.style={ at={(ticklabel* cs:1.025)}, anchor=south,}, 
			axis line style={->}] 
			\addplot[mark=none, draw=none, fill=gray!25, domain=0:12, samples=500] {1} \closedcycle; 
			\addplot[thin, color=gray!50, domain=0:12, samples=500] ({x},{1}); 
			\input{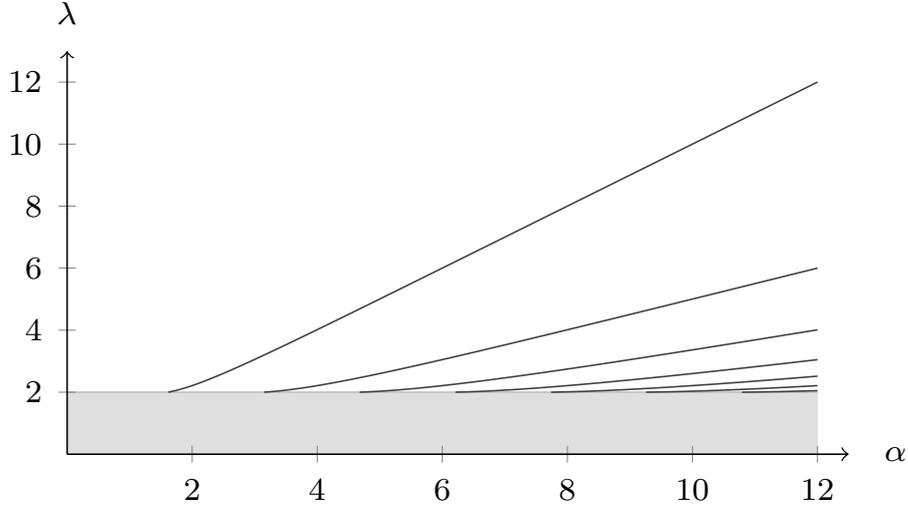} 
		\end{axis}
	\end{tikzpicture}
	\caption{Spectrum of $\alpha K_\alpha$ for $0< \alpha < 12$. The grey area is the a.c. spectrum and the black curves are the eigenvalues.}
	\label{fig:aKa} 
\end{figure}

A graphical representation of the spectrum of the rescaled matrix $\alpha K_\alpha$, obtained numerically, can be found in Figure~\ref{fig:aKa}.

\subsection{Connection with Jacobi matrices} \label{sec:jacint} A Jacobi matrix is a (possibly unbounded) operator on $\ell^2(\mathbb{N})$, given by the infinite tri-diagonal matrix of the form 
\begin{equation}
	J= 
	\begin{pmatrix}
		b(1)&a(1)&0&0&\cdots \\
		a(1)&b(2)&a(2)&0&\cdots \\
		0&a(2)&b(3)&a(3)&\cdots \\
		\vdots&\vdots&\vdots&\vdots&\ddots 
	\end{pmatrix}
	\label{eq:gen-jacobi} \end{equation}
with the \emph{Jacobi parameters} $b(j)$ real and $a(j)$ positive. Our approach to the spectral analysis of $K_\alpha$ relies on the crucial observation that $K_\alpha$ can be identified with an inverse of a Jacobi matrix $J_\alpha$ with the parameters 
\begin{align}\label{eq:aalpha} 
	a(n)&=a_\alpha(n)=\frac{n^{\alpha+\frac{1}{2}}(n+1)^{\alpha+\frac{1}{2}}}{n^{2\alpha}-(n+1)^{2\alpha}}, \\
	\label{eq:balpha} b(n)&=b_\alpha(n)=\frac{n^{2\alpha+1}((n+1)^{2\alpha}-(n-1)^{2\alpha})}{((n+1)^{2\alpha}-n^{2\alpha})(n^{2\alpha}-(n-1)^{2\alpha})}. 
\end{align}
Strictly speaking, $J_\alpha$ is the \emph{negative} of a Jacobi matrix, because our parameters $a_\alpha(n)$ are negative. Furthermore, we show that $k_\alpha$ is \emph{the only} continuous Hardy kernel such that the inverse of $K$ is a Jacobi matrix. 
\begin{theoremC}\label{thm.inv-jacobi} 
	\mbox{} 
	\begin{enumerate}
		\item[\rm (i)] Let $J_\alpha$ be the Jacobi matrix with the parameters \eqref{eq:aalpha} and \eqref{eq:balpha}, with the domain
		\[\Dom J_\alpha=\left\{x\in\ell^2\,:\, J_\alpha x\in\ell^2\right\}.\]
		Then $J_\alpha$ is self-adjoint and positive semi-definite. 
		\item[\rm (ii)] $J_\alpha$ is invertible with the inverse $K_\alpha$. 
		\item[\rm (iii)] Let $k(x,y)$ be a real-valued continuous function of $x>0$ and $y>0$, homogeneous of degree $-1$. Assume that the operator $K$ on $\ell^2(\mathbb{N})$, given by the infinite matrix $\{k(n,m)\}_{n,m=1}^\infty$, is bounded. Furthermore, assume that $K$ coincides with the inverse of a Jacobi matrix. Then $K=cK_\alpha$ for some $\alpha>0$ and $c<0$. 
	\end{enumerate}
\end{theoremC}

Of course, in view of the previous result, Theorems A and B can be rephrased as the statements about the spectral measure of the Jacobi matrix $J_\alpha$. For $\alpha=1/2$, this matrix corresponds to continuous dual Hahn polynomials, see e.g. \cite[Section~9.3]{KLS}. Specifically, the Jacobi matrix
\[-J_{1/2}+\frac{1}{4} I\]
corresponds to the recurrence relation \cite[Equation~9.3.4]{KLS} with parameters $a=1/2$, $b=1/2$ and $c=-1/2$. Thus, $J_{1/2}$ can be diagonalised explicitly, and its generalised eigenvectors are given in terms of the continuous dual Hahn polynomials. P. Otte \cite{Otte} observed that $K_{1/2}$ commutes with the Hilbert matrix $K_{\operatorname{c}}$, and, exploiting his observation, found a novel approach to the diagonalisation of $K_{\operatorname{c}}$.

To the best of our knowledge\footnote{\emph{Note added in proof.} F.~\v{S}tampach has informed us (private communication) that $J_1$ can be diagonalised by Wilson polynomials.}, for $\alpha\neq1/2$ the matrix $J_\alpha$ does not correspond to any known system of orthogonal polynomials.

\subsection{Failure of the reproducing kernel thesis} \label{sec:a4} Our interest in the family $K_\alpha$ arose due to its connection to the norms of certain composition operators on a Hilbert space of Dirichlet series; this connection was established in \cite[Section~2]{Brevig17}. In the present paper, we are able to use the spectral analysis of $K_\alpha$ to advance this line of research by proving the failure of the reproducing kernel thesis for composition operators.

To set the scene, let $\mathscr{H}^2$ denote the Hilbert space of Dirichlet series 
\begin{equation}\label{eq:diriseri} 
	f(s) = \sum_{n=1}^\infty x(n) n^{-s} 
\end{equation}
with square summable coefficients. It is easy to check that $\mathscr{H}^2$ is a space of analytic functions in the half-plane $\mathbb{C}_{1/2}$, where $\mathbb{C}_\theta = \left\{s = \sigma + it \,:\, \sigma>\theta\right\}$, and that the reproducing kernel of $\mathscr{H}^2$ at the point $w \in \mathbb{C}_{1/2}$ is $\mathscr{K}_w(s) = \zeta(s+\overline{w})$, i.e. for any $f\in \mathscr{H}^2$ we have
\[f(w)=\langle{f,\mathscr{K}_w\rangle}_{\mathscr{H}^2}.\]
Let $\varphi\colon \mathbb{C}_{1/2}\to\mathbb{C}_{1/2}$ be an analytic function (called \emph{symbol} in this context) such that the \emph{composition operator} $\mathscr{C}_\varphi f = f \circ \varphi$ maps $\mathscr{H}^2$ to itself and is a bounded operator on $\mathscr{H}^2$. The class of such symbols has been described in a seminal paper by Gordon and Hedenmalm \cite{GH99}, see Section \ref{sec:rpk} below for the definition.

The connection between composition operators and the matrix $K_\alpha$ appears through the analysis of the symbol 
\begin{equation}\label{eq:phialpha} 
	\varphi_\alpha(s) = \frac{1}{2}+\alpha \frac{1-2^{-s}}{1+2^{-s}}. 
\end{equation}
It turns out that $\|\mathscr{C}_{\varphi_\alpha} f\|_{\mathscr{H}^2}^2 = \left\langle K_\alpha x,\, x \right\rangle$, if $f$ is the Dirichlet series \eqref{eq:diriseri}; in other words, we have $K_\alpha=\mathscr{C}_{\varphi_\alpha}^\ast\mathscr{C}_{\varphi_\alpha}$ (with the standard identification between $\mathscr{H}^2$ and $\ell^2$). The proof of this claim relies on the integral representation \eqref{eq:intrep}, we refer again to \cite[Section~2]{Brevig17} for the details.

A natural question recently discussed by Muthukumar, Ponnusamy and Queff\'{e}lec in \cite[Section~5]{MPQ18}, is whether all bounded composition operators on $\mathscr{H}^2$ satisfy the \emph{reproducing kernel thesis}. This is the statement that the norm of $\mathscr{C}_\varphi$ can be evaluated by computing the action of $\mathscr{C}_\varphi$ on the set of reproducing kernels:
\[\sup_{w \in \mathbb{C}_{1/2}}\frac{\|\mathscr{C}_\varphi \mathscr{K}_w\|_{\mathscr{H}^2}}{\|\mathscr{K}_w\|_{\mathscr{H}^2}}=\|\mathscr{C}_\varphi\|.\]
Of course, the inequality $\leq$ here is always satisfied and the question is whether it is saturated on reproducing kernels. 

Our last main result gives a negative answer to this question. 
\begin{theoremD}
	The reproducing kernel thesis for the symbol \eqref{eq:phialpha} holds if and only if $N(K_\alpha)=0$ (see Theorem~B). In particular, the reproducing kernel thesis holds for all $0<\alpha\leq1$ and fails for all $2\leq\alpha<\infty$. 
\end{theoremD}
\begin{remark}
	The analogous question for composition operators on the Hardy space of the unit disc, $H^2(\mathbb{D})$, first raised by Cowen and MacCluer, was resolved in the negative by Appel, Bourdon and Thrall in \cite{ABT96}. In contrast to \cite{ABT96}, where the various quantities are explicitly estimated, we will rely on information about the spectrum of $K_\alpha$ to settle the reproducing kernel thesis for composition operators on $\mathscr{H}^2$. 
\end{remark}

\subsection{Structure of the paper} In Section~\ref{sec:jacobi}, we analyse the spectrum of the Jacobi matrices $J_\alpha$. We use the subordinacy theory of Gilbert and Pearson, which rests on the analysis of the asymptotics of the generalised eigenvectors of $J_\alpha$. These asymptotics are established in Section~\ref{sec:asymptotics}. The proof of Theorem~C is given in Sections~\ref{sec:sa}--\ref{sec:inv-jacobi}.

Section~\ref{sec:eigenvalues} contains the proof of Theorem~B. Part (i) is a consequence of the integral representation \eqref{eq:intrep}. The main idea of the proof of parts (ii) and (iii) is based on the fact that $K_\alpha$ converges entry-wise, as $\alpha\to\infty$, to a compact diagonal matrix with elements $1,\tfrac{1}{2},\frac{1}{3},\ldots$ on the diagonal. 

The final Section~\ref{sec:rpk} of the paper is devoted to the reproducing kernel thesis for composition operators on the Hardy space of Dirichlet series. 

\subsection{Related work} We would like to mention a few papers known to us where restrictions of (not necessarily Hardy) kernels onto the integer lattice were studied. The papers \cite{BPSSV,PP18} study the so-called multiplicative Hilbert matrix $\left(k(nm)\right)_{n,m=2}^\infty$ where $nm$ denotes multiplication and
\[k(x)=\frac1{\sqrt{x}\log(x)},\]
although the corresponding integral operator (which should in this case be considered on $L^2(1,\infty)$) does not play an important role, at least not explicitly. In \cite{MP18}, the authors study the similarly defined matrix with kernel
\[k(x)=\frac1{\sqrt{x}\log(x)(\log\log(x))^\alpha}, \qquad x\geq x_0>e,\]
by relating its spectral properties to those of the corresponding integral operator. In \cite{MP20}, some norm bounds relating $K$ and $\mathbf K$ were considered for Hankel kernels, i.e. $k(x,y)=h(x+y)$. In the very interesting paper \cite{KS} a family of integral operators $\mathbf{K}$ is related to a family of discrete operators $K$ in a different manner, yielding explicitly diagonalisable operators. In the short but inspirational paper \cite{Wilf}, the matrix $K_{1/2}$ is mentioned in passing and a connection to the theory of Dirichlet series (including formula \eqref{eq:intrep}) is outlined.

In \cite{Stampach22}, the family of infinite matrices $\bigl(1/(\max\{n,m\}+\nu)\bigr)_{n,m=0}^\infty$ is considered and diagonalised explicitly in terms of hypergeometric functions. For $\nu=2$, this matrix coincides with $K_{1/2}$. 

After this paper was completed, one of the present authors extended some aspects of our analysis to more general Hardy kernel matrices \cite{pushnitski2021}. The analysis of that paper indicates that it would be more accurate to compare the matrix $K$ not with the integral operator $\mathbf{K}$, but with the same integral operator acting on the space $L^2(1,\infty)$. Heuristically, this can be explained by the fact that restriction onto integers obliterates the singularity of the kernel $k(x,y)$ near $x=y=0$.

\subsection*{Acknowledgements} We are grateful to W. Van Assche and F. \v{S}tampach for helpful discussions. K.-M. Perfekt was supported by the UK Engineering and Physical Sciences Research Council (EPSRC), grant EP/S029486/1.

\section{The Jacobi matrices $J_\alpha$} \label{sec:jacobi}

\subsection{Warm-up: the continuous analogue of $J_\alpha$} The purpose of this section is to analyse the spectrum of the Jacobi matrices $J_\alpha$ with the Jacobi parameters \eqref{eq:aalpha} and \eqref{eq:balpha}. In order to suggest some intuition into this analysis, we start by briefly discussing the continuous analogue of $J_\alpha$. In this preliminary discussion we omit the proofs. 

Along with $\mathbf{K}_\alpha$, let us consider the differential operator $\mathbf{J}_\alpha$ in $L^2(\mathbb{R}_+)$ given by 
\begin{equation}\label{eq:boldJ} 
	\begin{split}
		\mathbf{J}_\alpha &=-\frac{1}{2\alpha}\left(x\frac{d}{dx}+\frac{1}{2}-\alpha\right)\left(x\frac{d}{dx}+\frac{1}{2}+\alpha\right) \\
		&=-x^{\alpha+\frac{1}{2}}\frac{d}{dx}\frac{1}{2\alpha\, x^{2\alpha-1}}\frac{d}{dx}x^{\alpha+\frac{1}{2}} 
	\end{split}
\end{equation}
with the domain
\[\Dom \mathbf{J}_\alpha=\left\{ f\in L^2(\mathbb{R}_+): xf', x^2f''\in L^2(\mathbb{R}_+)\right\}.\]
One can check that $\mathbf{J}_\alpha$ is self-adjoint and positive semi-definite. It is not difficult to prove that $\mathbf{J}_\alpha$ is the inverse of $\mathbf{K}_\alpha$:
\[\mathbf{J}_\alpha\mathbf{K}_\alpha=\mathbf{K}_\alpha\mathbf{J}_\alpha=\mathbf{I}.\]
In particular, the spectrum of $\mathbf{J}_\alpha$ is purely absolutely continuous, 
\begin{equation}\label{eq:Jspec} 
	\sigma_{\ac}(\mathbf{J}_\alpha)=[\tfrac{\alpha}{2},\infty) \quad \text{with multiplicity \underline{two}.} 
\end{equation}
Observe that $\mathbf{J}_\alpha$ also commutes with dilations and so it is diagonalised by the Mellin transform. In fact, by the Mellin transform $\mathbf{J}_\alpha$ is unitarily equivalent to the operator of multiplication by the function 
\begin{equation}\label{eq:1/omega} 
	\frac{1}{\omega_\alpha(t)}=\frac{\alpha^2+t^2}{2\alpha}, \qquad t\in \mathbb{R}, 
\end{equation}
where $\omega_\alpha$ is as in \eqref{eq:omega}.

In the rest of this section, we prove suitable analogues of these facts for the Jacobi matrix $J_\alpha$: self-adjointness, the analogue of factorisation \eqref{eq:boldJ}, the relation $J_\alpha^{-1}=K_\alpha$ and a suitable substitute for \eqref{eq:Jspec} (Theorem~\ref{thm:Jspec} below). Despite many similarities, there are some important differences between the discrete and continuous cases: 
\begin{itemize}
	\item No explicit diagonalisation of $J_\alpha$ is available (except for $\alpha=1/2$). 
	\item In the continuous case, the operators $\mathbf{J}_\alpha$ are, for \emph{all} $\alpha$, diagonalised by the Mellin transform. This means, in particular, that operators $\mathbf{J}_{\alpha_1}$ and $\mathbf{J}_{\alpha_2}$ commute for any $\alpha_1$ and $\alpha_2$. It is not too difficult to see that this commutation property is false in the discrete case. 
	\item In the continuous case, the spectrum of $\mathbf{J}_\alpha$ is purely absolutely continuous. We will see that the spectrum of $J_\alpha$ has some eigenvalues for large $\alpha$. 
	\item The spectrum of $\mathbf{J}_\alpha$ has multiplicity two, while the spectrum of $J_\alpha$ has multiplicity one. 
\end{itemize}

\subsection{Factorisation of $J_\alpha$} For $n\in\mathbb{N}$, let 
\begin{equation}\label{eq:caalphan} 
	c_\alpha(n) = \frac{1}{n^{2\alpha}-(n-1)^{2\alpha}}. 
\end{equation}
We observe that the coefficients $a_\alpha$ and $b_\alpha$ defined in \eqref{eq:aalpha} and \eqref{eq:balpha}, can be written as 
\begin{align}
	a_\alpha(n) &= -n^{\alpha+\frac{1}{2}}(n+1)^{\alpha+\frac{1}{2}}c_\alpha(n+1), \label{eq:aalphan} \\
	b_\alpha(n) &= n^{2\alpha+1}(c_\alpha(n)+c_\alpha(n+1)). 
\label{eq:balphan} \end{align}
If $\gamma=\gamma(n)$ is a sequence of real numbers, we will denote by $D(\gamma)$ the infinite diagonal matrix with elements $\gamma(1),\gamma(2),\ldots$ on the diagonal. In other words, $D(\gamma)$ is the operator of multiplication by the sequence $\gamma(n)$ in $\ell^2$. By a slight abuse of notation, we will write $D(n^\beta)$ if $\gamma(n)=n^\beta$ for all $n\geq1$. We will also use the usual shift operator $S$ (and its adjoint $S^\ast$) in $\ell^2$, defined by $Se_n = e_{n+1}$, where $e_n$ denotes the standard basis vectors of $\ell^2$. 

The following result should be compared to \eqref{eq:boldJ}. 
\begin{lemma}\label{lem:factor} 
	The matrix $J_\alpha$ can be represented as 
	\begin{equation}\label{eq:factorisation} 
		J_\alpha = D(n^{\alpha+\frac{1}{2}})(I-S^\ast)D(c_\alpha)(I-S) D(n^{\alpha+\frac{1}{2}}). 
	\end{equation}
\end{lemma}
\begin{proof}
	We first note that from \eqref{eq:aalphan}, \eqref{eq:balphan} it is clear that we can rewrite
	\[J_\alpha = D(n^{\alpha+\frac{1}{2}})J'_\alpha D(n^{\alpha+\frac{1}{2}}),\]
	where
	\[J'_\alpha = 
	\begin{pmatrix}
		c_\alpha(1)+c_\alpha(2) & -c_\alpha(2) & 0 & 0 & \cdots \\
		-c_\alpha(2) & c_\alpha(2)+c_\alpha(3) & -c_\alpha(3) & 0 & \cdots \\
		0 & -c_\alpha(3) & c_\alpha(3)+c_\alpha(4) & -c_\alpha(4) & \cdots \\
		\vdots& \vdots & \vdots & \vdots & \ddots 
	\end{pmatrix}
	.\]
	Next, observe that if $A$ is any infinite matrix, then $S^\ast A$ is the same matrix ``shifted up'', and $AS$ is the same matrix ``shifted left''. Using this observation, it is immediate that
	\[J'_\alpha=D(c_\alpha)+S^\ast D(c_\alpha)S-D(c_\alpha)S-S^\ast D(c_\alpha).\]
	This can be more succinctly written as
	\[J'_\alpha= (I-S^\ast) D(c_\alpha)(I-S).\]
	Coming back to $J_\alpha$, we obtain the factorisation \eqref{eq:factorisation} as desired. 
\end{proof}

\subsection{Proof of Theorem C (i)} \label{sec:sa} We begin by briefly recalling the corresponding general framework; see e.g. \cite[Section~2.6]{Teschl} for details. 

Let $J$ be a Jacobi matrix as in \eqref{eq:gen-jacobi}. Obviously, for each $x\in\ell^2$ we can define the usual matrix product $Jx$, but it does not have to be in $\ell^2$. Let $\ell_0^2\subset\ell^2$ be the dense subset of finitely supported sequences, i.e. $x\in\ell_0^2$ if $x_n=0$ for all sufficiently large $n$. One can define two operators in $\ell^2$ associated with $J$: 
\begin{itemize}
	\item $J_{\text{min}}$ is the closure of the operator $J$ defined on the ``minimal domain'' $\ell_0^2$; 
	\item $J_{\text{max}}$ is the operator with the ``maximal domain''
	\[\Dom J_{\text{max}}=\left\{x\in\ell^2: Jx\in\ell^2\right\}.\]
\end{itemize}
The operator $J_{\text{min}}$ is symmetric and $J_{\text{min}}^\ast=J_{\text{max}}$, but it may happen that $J_{\text{min}}\not=J_{\text{max}}$ and then neither of these two operators is self-adjoint. We have $J_{\text{min}}=J_{\text{max}}$ if and only if $J$ is the \emph{limit point case}; this means the following. For $\lambda\in\mathbb{C}$, consider the set of solutions $x$ to the three-term recurrence relation 
\begin{equation}\label{eq:rec-rel} 
	a(n-1)x(n-1)+b(n)x(n)+a(n)x(n+1)=\lambda x(n), \qquad n\geq2. 
\end{equation}
It is important to note that the first equation
\[b(1)x(1)+a(1)x(2)=\lambda x(1)\]
is not required to hold; in the language of ODEs, this equation plays the role of the boundary condition for $x$. Obviously, for each $\lambda\in\mathbb{C}$, there are exactly two linearly independent solutions to \eqref{eq:rec-rel}. 

A Jacobi matrix $J$ is called a \emph{limit point case}, if for some $\lambda\in\mathbb{C}$ at least one of the linearly independent solutions to \eqref{eq:rec-rel} is not in $\ell^2$. It is not difficult to see that if this is true for some $\lambda\in\mathbb{C}$, then it is true for all $\lambda\in\mathbb{C}$. If $J$ is not a limit point case, it is called a \emph{limit circle case}. 

If $J$ is a limit point case, then we have $J_{\text{max}}=J_{\text{min}}$ and in this case $J_{\text{max}}$ is self-adjoint. The equality $J_{\text{max}}=J_{\text{min}}$ means that $\ell_0^2$ is dense in the maximal domain $\Dom J_{\text{max}}$ with respect to the graph norm $(\|x\|^2+\|Jx\|^2)^{1/2}$. 

By using the factorisation \eqref{eq:factorisation}, it is easy to see that the solutions to the recurrence relation \eqref{eq:rec-rel} for $\lambda=0$ are of the form
\[x(n)=c_1 n^{-\frac{1}{2}-\alpha}+c_2 n^{-\frac{1}{2}+\alpha},\]
so there is exactly one linearly independent $\ell^2$-solution. Thus, $J_\alpha$ is the limit point case and therefore the corresponding operator on $\ell^2$ with the domain as stated in the hypothesis of the theorem is self-adjoint. 

Let us prove that $J_\alpha$ is positive semi-definite. Fix $x\in \Dom J_\alpha$ and let $x^{(N)}\in\ell_0^2$ be a sequence approximating $x$ in the graph norm of $J_\alpha$. We have
\[\langle J_\alpha x,\,x \rangle =\lim_{N\to\infty}\big\langle J_\alpha x^{(N)},\,x^{(N)}\big\rangle.\]
Since each $x^{(N)}$ is finitely supported, we can rearrange the order of summation, which yields 
\begin{align*}
	\big\langle J_\alpha x^{(N)},\,x^{(N)}\big\rangle &= \big\langle D(n^{\alpha+\frac12})(I-S^\ast)D(c_\alpha)(I-S)D(n^{\alpha+\frac12})x^{(N)},\,x^{(N)}\big\rangle \\
	&= \big\langle D(c_\alpha)(I-S)D(n^{\alpha+\frac12})x^{(N)},\,(I-S)D(n^{\alpha+\frac12})x^{(N)}\big\rangle \\
	&= \big\|D(\sqrt{c_\alpha})(I-S)D(n^{\alpha+\frac12})x^{(N)}\big\|^2\geq0. 
\end{align*}
It follows that $\langle J_\alpha x,\,x\rangle\geq0$, as required. \qed

\subsection{Proof of Theorem~C (ii)} First we observe that 
\begin{equation}\label{a2} 
	K_\alpha=D(n^{-\alpha-\frac{1}{2}})M_\alpha D(n^{-\alpha-\frac{1}{2}}), 
\end{equation}
where $M_\alpha$ is the infinite matrix
\[M_\alpha = \left([\min(n,m)]^{2\alpha}\right)_{n,m\geq1}.\]
Note that $M_\alpha$ has constant elements on the ``infinite corners'' $\min(n,m)=\text{const}$. 

Next, we claim that 
\begin{equation}\label{a1} 
	(I-S)M_\alpha (I-S^\ast)=D(n^{2\alpha})-D((n-1)^{2\alpha}). 
\end{equation}
This can be understood, for example, as an identity on $\ell_0^2$. To see this, consider the left hand side,
\[M_\alpha+SM_\alpha S^\ast -S M_\alpha- M_\alpha S^\ast.\]
Observe, as in the proof of Lemma~\ref{lem:factor}, that $S M_\alpha$ is the matrix $M_\alpha$ ``shifted down'' (with a zero first row) and $M_\alpha S^\ast $ is the matrix $M_\alpha$ ``shifted right'' (with a zero first column). Because $M_\alpha$ has constant elements on ``infinite corners'', the off-diagonal elements in the above combination vanish, and the diagonal elements are easy to work out, yielding \eqref{a1}. 

To prove that $J_\alpha$ is invertible with inverse $K_\alpha$, it is sufficient to show that $J_\alpha K_\alpha=I$, since both operators are self-adjoint and $K_\alpha$ is bounded. Since $J_\alpha$ is closed and $K_\alpha$ is bounded, it is sufficient to prove that $J_\alpha K_\alpha x = x$ for all $x\in\ell_0^2$. According to the factorisations \eqref{eq:factorisation} and \eqref{a2}, we need to verify the identity
\[D(n^{\alpha+\frac{1}{2}})(I-S^\ast)D(c_\alpha)(I-S) M_\alpha D(n^{-\alpha-\frac{1}{2}})x=x,\]
or, making the change of variable $y=D(n^{-\alpha-\frac{1}{2}})x$,
\[(I-S^\ast )D(c_\alpha)(I-S)M_\alpha y =y.\]
Since any $y\in\ell_0^2$ can be written as $y = (I - S^\ast)z$ for a unique $z\in\ell_0^2$, it suffices to check that
\[D(c_\alpha)(I-S)M_\alpha(I-S^\ast ) z= z.\]
But the last identity follows directly from \eqref{a1} and the definition of the sequence $c_\alpha$ in \eqref{eq:caalphan}. \qed

\subsection{Proof of Theorem~C (iii)} \label{sec:inv-jacobi}

By a well-known explicit calculation, the matrix of a (bounded) inverse $K$ of a Jacobi matrix $J$ can be expressed by the formula
\[K(n,m)=\varphi_-(\min(n,m))\varphi_+(\max(n,m)), \qquad n,m\in\mathbb{N}.\]
Here and in what follows $\varphi_+$, $\varphi_-$ are certain non-zero solutions to the system of recurrence relations 
\begin{equation}
	a(n-1)\varphi_{\pm}(n-1)+b(n)\varphi_\pm(n)+a(n)\varphi_{\pm}(n+1)=0,\qquad n\geq2. 
\label{z1} \end{equation}
We will not need to know anything about these solutions, but for completeness we mention that $\varphi_-$ satisfies \eqref{z1} with $n=1$ (the ``boundary condition at the origin''), $\varphi_+$ satisfies $\varphi_+\in\ell^2(\mathbb{N})$ (the ``boundary condition at infinity'') and they are normalised to satisfy the ``discrete Wronskian'' condition
\[a(n)\bigl(\varphi_+(n+1)\varphi_-(n)-\varphi_+(n)\varphi_-(n+1)\bigr)=1.\]
Let us write
\[k(x,y)=\frac1{\sqrt{xy}}h(x/y),\]
where, by our assumption, $h$ is a continuous function on $(0,\infty)$. Then we have
\[\frac{1}{\sqrt{nm}}h(n/m)=\varphi_-(\min(n,m))\varphi_+(\max(n,m));\]
denoting $\psi_\pm(n)=\sqrt{n}\varphi_\pm(n)$, this rewrites as
\[h(n/m)=\psi_-(\min(n,m))\psi_+(\max(n,m)).\]
Setting $n=m$, we find that $h(1)=\psi_-(n)\psi_+(n)$. In particular, we must have $h(1)\not=0$ (as both $\varphi_+$ and $\varphi_-$ are not identically zero). It follows that $\psi_+(n)\not=0$ and $\psi_-(n)\not=0$ for all $n$ and so, denoting $h_0(x)=h(x)/h(1)$, 
\begin{equation}\label{z2} 
	h_0(n/m)=\frac{\psi_+(\max(n,m))}{\psi_+(\min(n,m))}. 
\end{equation}
For $m=1$ this yields
\[h_0(n)=\frac{\psi_+(n)}{\psi_+(1)}, \qquad n\geq1.\]
Substituting this back into \eqref{z2}, we find 
\begin{equation}\label{z4} 
	h_0(n/m)=\frac{h_0(\max(n,m))}{h_0(\min(n,m))}. 
\end{equation}
If $n=km$ with $k,m\in\mathbb{N}$, we find
\[h_0(km)=h_0(k)h_0(m),\]
i.e. $h_0$ is completely multiplicative. In order to work effectively with the completely multiplicative function $h_0$, we need to introduce some (standard) notation. Let $\{p_j\}_{j=1}^\infty$ be the ordered sequence of all prime numbers, so that $p_1=2$, $p_2=3$, $p_3=5$ etc. Any natural number $n$ can be written as a product $n=p_1^{\varkappa_1}p_2^{\varkappa_2}\cdots$, where $\varkappa_j$ are non-negative integers for all $j$. We use the shorthand notation $n=p^\varkappa$ for this product. 

Denote $r_j=h_0(p_j)$ for all $j$; then by the complete multiplicativity, we have
\[h_0(p^\varkappa)=r_1^{\varkappa_1}r_2^{\varkappa_2}\cdots=:r^\varkappa.\]
Substituting this into \eqref{z4}, we find 
\begin{equation}\label{z5} 
	h_0(p^\varkappa/p^\nu)= 
	\begin{cases}
		r^\varkappa/r^\nu & \text{ if $p^\varkappa\geq p^\nu$,} \\
		r^\nu/r^\varkappa & \text{ if $p^\varkappa<p^\nu$.} 
	\end{cases}
\end{equation}
Now let us pick two distinct primes, for example 2 and 3. Consider two sequences of natural numbers $j_n$ and $k_n$ such that 
\begin{equation}\label{z6} 
	\lim_{n\to \infty} \big(j_n \log{2}-k_n\log{3}\big) = 0. 
\end{equation}

Then we have $2^{j_n}/3^{k_n} \to 1$ as $n\to\infty$ and so, by the continuity of $h_0$,
\[\lim_{n\to \infty} h_0(2^{j_n}3^{-k_n})=1.\]
By \eqref{z5}, this implies that $r_1^{j_n}/r_2^{k_n} \to 1$ as $n\to \infty$. From here and \eqref{z6} we deduce that
\[\frac{\log r_1}{\log 2}=\frac{\log r_2}{\log 3}.\]
As this argument is applicable to every pair of primes, we find that the ratio
\[\frac{\log r_j}{\log p_j}\]
is independent of $j$. Denote this ratio by $-\alpha$; then we get $r_j=p_j^{-\alpha}$ and so \eqref{z5} simplifies to
\[h_0(q)= 
\begin{cases}
	q^{-\alpha}, & q\geq1, \\
	q^{\alpha}, & q\leq 1 
\end{cases}
\]
for all positive rational $q$, or equivalently
\[h_0(q)=\bigl(\min(q,q^{-1})\bigr)^{\alpha}.\]
Returning to our original notation, this gives
\[k(n,m)=\frac{h(1)}{\sqrt{nm}}\bigl(\min(n/m,m/n)\bigr)^\alpha.\]
The boundedness of $K$ necessitates $\alpha>0$, and we get $k(n,m)=h(1)k_\alpha(n,m)$. Since $K^{-1}$ is a Jacobi matrix (with the normalisation $a(j)>0$), we must have $h(1)=c<0$. \qed

\subsection{Asymptotics for generalised eigenvectors} Our next aim is to prove Theorem~\ref{thm:Jspec}, which describes the structure of the spectrum of $J_\alpha$. A crucial step in its proof is the analysis of the asymptotics of the generalised eigenvectors of $J_\alpha$. These are understood as the solutions $x$ to the recurrence relation 
\begin{equation}\label{eq:diffeq} 
	a_\alpha(n-1)x(n-1)+b_\alpha(n)x(n)+a_\alpha(n)x(n+1)=\lambda x(n), \qquad n\geq2; 
\end{equation}
as in Section~\ref{sec:sa}, we do not require $x$ to be in $\ell^2$ and we do not require the first equation 
\begin{equation}\label{c2} 
	b_\alpha(1)x(1)+a_\alpha(1)x(2)=\lambda x(1) 
\end{equation}
to hold. The proof of the following result is postponed to Section~\ref{sec:asymptotics}. 
\begin{lemma}\label{lem1} 
	Fix $0<\alpha<\infty$ and let $s$ be a complex number with $\mre{s}\geq0$. Set
	\[\lambda=\frac{\alpha^2-s^2}{2\alpha}.\]
	There exists a unique solution $x=x_{\alpha,s}$ of the recurrence relation \eqref{eq:diffeq} such that 
	\begin{equation}\label{eq:xasymp} 
		x_{\alpha,s}(n)=n^{-\frac{1}{2}-s}+O_{\alpha,s}(n^{-\frac{1}{2}-\mre{s}-2}), 
	\end{equation}
	as $n\to\infty$. For each positive integer $n$, the value $x_{\alpha,s}(n)$ extends to an analytic function in $s$ in the half-plane $\mre{s}>-1$. If $s=0$, there exists another solution $x'_{\alpha,0}$ to \eqref{eq:diffeq} with the asymptotics 
	\begin{equation}\label{eq:hereislog} 
		x'_{\alpha,0}(n) = \frac{\log{n}}{\sqrt{n}} + O_\alpha(n^{-\frac{3}{2}}\log{n}). 
	\end{equation}
\end{lemma}
\begin{remark}
	The generalised eigenvectors of the differential operator $\mathbf{J}_\alpha$ are \emph{exactly} the functions $x^{-\frac{1}{2}\pm s}$, where $s=it$, $t\in\mathbb R$ and $s$ is related to $\lambda$ by
	\[\lambda=\frac{\alpha^2-s^2}{2\alpha}=\frac{\alpha^2+t^2}{2\alpha},\]
	see \eqref{eq:1/omega}. It is instructive to compare this with the asymptotics \eqref{eq:xasymp}. 
\end{remark}

\subsection{Subordinacy theory and the spectrum of $J_\alpha$} Here we prove the following theorem characterising the spectrum of $J_\alpha$. 
\begin{theorem}\label{thm:Jspec} 
	For $0<\alpha<\infty$, the a.c. spectrum of $J_\alpha$ is
	\[\sigma_{\ac}(J_\alpha)=[\tfrac{\alpha}{2},\infty) \quad \text{with multiplicity \underline{one}.}\]
	The singular continuous spectrum of $J_\alpha$ is empty. $J_\alpha$ has finitely many (possibly none) eigenvalues, all of which are simple and located in the interval $(0,\tfrac{\alpha}{2})$. 
\end{theorem}

Since $K_\alpha=J_\alpha^{-1}$, this immediately yields Theorem~A. Of course, the statements about the simplicity of the spectrum are obvious, as all Jacobi matrices have simple spectrum. 

Our main tool in the proof of Theorem~\ref{thm:Jspec} is the subordinacy theory of Gilbert and Pearson \cite{GP}, adapted to the case of Jacobi matrices in \cite{KP92}. A solution $u$ to the recurrence relation \eqref{eq:diffeq} is called \emph{subordinate}, if for any other linearly independent solution $v$ we have
\[\lim_{N\to\infty}\frac{\sum_{n=1}^N |u(n)|^2}{\sum_{n=1}^N |v(n)|^2}=0.\]
Subordinacy theory applies to Jacobi matrices in the limit point case (see the remark before Lemma~\ref{lem1}). Note that in this case, if we have a solution $u\in \ell^2$, it is always subordinate, because any other linearly independent solution will not be in $\ell^2$. 

Specifically, we will use \cite[Theorem~3]{KP92} which states that 
\begin{itemize}
	\item The absolutely continuous spectrum of $J_\alpha$ coincides (up to sets of measure zero) with the set of $\lambda\in\mathbb{R}$ where no subordinate solution of \eqref{eq:diffeq} exists; 
	\item The singular spectrum of $J_\alpha$ (i.e. the union of the point spectrum and the singular continuous spectrum) coincides with the set of $\lambda\in \mathbb{R}$ where a subordinate solution of \eqref{eq:diffeq} exists and satisfies the initial condition \eqref{c2}. 
\end{itemize}
\begin{proof}[Proof of Theorem~\ref{thm:Jspec}] Let us inspect the asymptotics of Lemma~\ref{lem1} with $\lambda\in \mathbb{R}$. We have three cases: 
	
	\emph{Case 1:} $s=it$ for $t \in \mathbb{R}\backslash\{0\}$. Then $\lambda=(\alpha^2+t^2)/2\alpha$, so $\lambda>\alpha/2$. In this case we have two linearly independent solutions with asymptotics
	\[x_{\alpha,s}(n)=n^{-\frac{1}{2}\pm it}+O_{\alpha,s}(n^{-\frac{5}{2}})\]
	as $n\to\infty$. From here it is clear that no subordinate solutions exist. It follows that the absolutely continuous spectrum of $J_\alpha$ contains the interval $(\alpha/2,\infty)$, and that there is no singular spectrum on this interval. 
	
	\emph{Case 2:} $s=\sigma$ for $\sigma>0$. Then $\lambda=(\alpha^2-\sigma^2)/2\alpha$, so $\lambda<\alpha/2$. As we already know that $J_\alpha$ is positive definite and invertible, we are only interested in the range $0<\lambda<\alpha/2$. In this case we have a solution with the asymptotics
	\[x_{\alpha,s}(n)=n^{-\frac{1}{2}-\sigma}+O_{\alpha,s}(n^{-\frac{1}{2}-\sigma-2}),\]
	as $n\to\infty$. Clearly, this solution is in $\ell^2$ and therefore it is subordinate. It follows that the interval $(0,\alpha/2)$ contains no absolutely continuous spectrum of $J_\alpha$ and that a given $\lambda$ in this interval is an eigenvalue if and only if the above subordinate solution satisfies the boundary condition \eqref{c2}, viz.
	\[b_\alpha(1)x_{\alpha,s}(1)+a_\alpha(1)x_{\alpha,s}(2)-\frac{\alpha^2-s^2}{2\alpha} x_{\alpha,s}(1)=0.\]
	By Lemma~\ref{lem1} the left hand side here is an analytic function of $s$ for $\mre {s}>-1$. Therefore, there can only be finitely many zeros of this function in the range of $s$ corresponding to $0<\lambda<\alpha/2$, and so $J_\alpha$ can only have finitely many eigenvalues. 
	
	\emph{Case 3:} Finally, consider $\lambda=\alpha/2$; this corresponds to the case $s=0$ in \eqref{eq:xasymp}. In this case we also have another linearly independent solution \eqref{eq:hereislog}. None of these two solutions is in $\ell^2$, so $\lambda=\alpha/2$ is not an eigenvalue. 
\end{proof}

\section{Asymptotics for generalised eigenvectors of $J_\alpha$} \label{sec:asymptotics}

\subsection{Preliminaries} The goal of the present section is to prove Lemma~\ref{lem1}. Our approach is heavily based on the work of Wong and Li \cite{WL92} on the asymptotic behaviour of solutions to second degree difference equation of which \eqref{eq:diffeq} is a special case. In fact, the existence and uniqueness of the solution \eqref{eq:xasymp} is a direct consequence of \cite{WL92}. We choose to reproduce their arguments rather than to simply quote their result for two reasons: 
\begin{itemize}
	\item Both the factorisation \eqref{eq:factorisation} of $J_\alpha$ and the explicit formula for the inverse $J_\alpha^{-1}=K_\alpha$ allow us to simplify and streamline the arguments of \cite{WL92}; 
	\item In order to prove the analyticity of $x_{\alpha,s}(n)$, we will need to keep track of certain coefficients throughout the proof, which was not the focus of \cite{WL92}. 
\end{itemize}
One exception is the special case $\lambda=\alpha/2$ where we will for simplicity refer to \cite{WL92} for the proof of \eqref{eq:hereislog}. 

\subsection{Construction of approximate solutions} Throughout this section, it will be convenient to use the notation $\zeta_s$ for the ``standard'' sequence 
\begin{equation}\label{eq:zs} 
	\zeta_s(n)=n^{-s}, \qquad n\in\mathbb{N}. 
\end{equation}
We begin with the following key estimate. 
\begin{lemma}\label{lem:diagonal} 
	Let $\beta\in\mathbb{C}$; if $n\geq2$, then 
	\begin{equation}\label{eq:diagonal} 
		(J_\alpha \zeta_{\frac{1}{2}+\beta})(n)= \frac{\alpha^2-\beta^2}{2\alpha} \zeta_{\frac{1}{2}+\beta}(n) + \sum_{j=1}^\infty C_{\alpha,\beta}(j) \zeta_{\frac{1}{2}+\beta+2j}(n), 
	\end{equation}
	where for fixed $\alpha$ and $j$ the function $\beta \mapsto C_{\alpha,\beta}(j)$ is a polynomial of degree at most $2j+2$, and for fixed $\alpha$ it holds that $C_{\alpha,\beta}(j) = O_{\alpha}((3/2)^j)$, locally uniformly in $\beta$. 
\end{lemma}
\begin{remark}
	From the estimate on $C_{\alpha,\beta}(j)$ it follows that \eqref{eq:diagonal} can be interpreted as the asymptotic series as $n\to\infty$, i.e. for any $k\geq1$,
	\[(J_\alpha \zeta_{\frac{1}{2}+\beta})(n)=\frac{\alpha^2-\beta^2}{2\alpha} \zeta_{\frac{1}{2}+\beta}(n) + \sum_{j=1}^{k-1} C_{\alpha,\beta}(j) \zeta_{\frac{1}{2}+\beta+2j}(n) +O_{\alpha,\beta}(\zeta_{\frac{1}{2}+\beta+2k}(n)).\]
\end{remark}
\begin{proof}
	We begin with the expression
	\[(J_\alpha \zeta_{\frac{1}{2}+\beta})(n) = (n-1)^{-\frac{1}{2}-\beta}a_\alpha(n-1) + n^{-\frac{1}{2}-\beta} b_\alpha(n) + (n+1)^{-\frac{1}{2}-\beta} a_\alpha(n).\]
	Recalling \eqref{eq:aalphan} and \eqref{eq:balphan} we find that 
	\begin{align*}
		(n-1)^{-\frac{1}{2}-\beta}a_\alpha(n-1) &= -n^{\frac{1}{2}-\beta} \frac{(1-n^{-1})^{\alpha-\beta}}{1-(1-n^{-1})^{2\alpha}}, \\
		n^{-\frac{1}{2}-\beta} b_\alpha(n) &= n^{\frac{1}{2}-\beta}\left(\frac{1}{1-(1-n^{-1})^{2\alpha}}-\frac{1}{1-(1+n^{-1})^{2\alpha}}\right), \\
		(n+1)^{-\frac{1}{2}-\beta} a_\alpha(n) &= n^{\frac{1}{2}-\beta} \frac{(1+n^{-1})^{\alpha-\beta}}{1-(1+n^{-1})^{2\alpha}}, 
	\end{align*}
	from which we conclude that
	\[(J_\alpha \zeta_{\frac{1}{2}+\beta})(n) = n^{\frac{1}{2}-\beta} \left(\frac{1-(1-n^{-1})^{\alpha-\beta}}{1-(1-n^{-1})^{2\alpha}}-\frac{1-(1+n^{-1})^{\alpha-\beta}}{1-(1+n^{-1})^{2\alpha}}\right).\]
	Hence 
	\begin{equation}\label{eq:c1} 
		(J_\alpha \zeta_{\frac{1}{2}+\beta})(n) = \zeta_{\frac{1}{2}+\beta}(n)f_{\alpha,\beta}(1/n), 
	\end{equation}
	where
	\[f_{\alpha,\beta}(z) = \frac{1}{z}\left(\frac{1-(1-z)^{\alpha-\beta}}{1-(1-z)^{2\alpha}}-\frac{1-(1+z)^{\alpha-\beta}}{1-(1+z)^{2\alpha}}\right). \]
	By inspection, $f_{\alpha,\beta}$ is even, and so its Taylor expansion has the form
	\[f_{\alpha,\beta}(z)=\sum_{j=0}^\infty C_{\alpha,\beta}(j) z^{2j}.\]
	Computing $f_{\alpha,\beta}(0)$, we get
	\[f_{\alpha,\beta}(z)=\frac{\alpha^2-\beta^2}{2\alpha}+ \sum_{j=1}^\infty C_{\alpha,\beta}(j) z^{2j}.\]
	Since $f_{\alpha,\beta}$ is analytic in $|z|<1$ and depends analytically on $\beta$, it holds that $C_{\alpha,\beta}(j) \leq C_{\alpha,\delta}(1+\delta)^j$ for any $\delta>0$, locally uniformly in $\beta$. For us $\delta=1/2$ will suffice. It is also clear that for fixed $\alpha$ and $j$, the function $\beta \mapsto C_{\alpha,\beta}(j)$ is a polynomial of degree at most $2(j+1)$. The proof is complete. 
\end{proof}

We can now easily construct approximate solutions to the eigenvalue equation, such that the error sequence $u=(J_\alpha-\lambda) y$ has arbitrarily fast polynomial decay as $n\to\infty$. 
\begin{lemma}\label{lem:induction} 
	Suppose that $\mre{s}>-1$ and set $\lambda = (\alpha^2-s^2)/2\alpha$. For each positive integer $k$ there is a sequence $y_k$ of the form 
	\begin{equation}\label{eq:goodguy} 
		y_k=\sum_{j=0}^{k-1} Y_{\alpha,s}(j) \zeta_{\frac{1}{2}+s+2j}, 
	\end{equation}
	normalised by $Y_{\alpha,s}(0)=1$, such that for each $n\geq2$ it holds that 
	\begin{equation}\label{eq:induction} 
		((J_\alpha-\lambda) y_k)(n) = \sum_{j=k}^\infty C_{\alpha,s,k}(j)\zeta_{\frac{1}{2}+s+2j}(n). 
	\end{equation}
	Here $C_{\alpha,s,k}(j) = O_{\alpha,k}((3/2)^j)$, locally uniformly in $s$. Moreover: 
	\begin{itemize}
		\item For fixed $\alpha$ and $j\geq1$, the function $s \mapsto Y_{\alpha,s}(j)$ is a rational function with simple poles at $s=-1, -2, \ldots, -j$. 
		\item The function $s \mapsto C_{\alpha,s,1}(j)$ is a polynomial in $s$ for every $j\geq1$. 
		\item For $k\geq2$ the function $s \mapsto C_{\alpha,s,k}(j)$ is a rational function with simple poles at $s=-1, -2, \ldots, -(k-1)$ for every $j\geq k$. 
	\end{itemize}
\end{lemma}
\begin{proof}
	We argue by induction in $k\geq1$. The case $k=1$ follows directly from Lemma~\ref{lem:diagonal}. Indeed, we set $Y_{\alpha,s}(0)=1$ and obtain \eqref{eq:induction} for $k=1$ with $C_{\alpha,s,1}(j)=C_{\alpha,s}(j)$, where the latter coefficient is from \eqref{eq:diagonal} with $\beta=s$. 
	
	Suppose now that for some fixed $k\geq1$ we have a sequence $y_k$ of the form \eqref{eq:goodguy} which satisfies \eqref{eq:induction}. We choose
	\[y_{k+1} = y_k - \frac{\alpha C_{\alpha,s,k}(k)}{2 k(s+k)} \zeta_{\frac{1}{2}+s+2k},\]
	where $C_{\alpha,s,k}(k)$ is the first coefficient in the expansion \eqref{eq:induction}. The plan is to use Lemma~\ref{lem:diagonal} with $\beta=s+2k$. Note that
	\[\frac{\alpha^2-\beta^2}{2\alpha}-\lambda=\frac{\alpha^2-(s+2k)^2}{2\alpha}-\frac{\alpha^2-s^2}{2\alpha} = \frac{2 k(s+k)}{\alpha}.\]
	Hence, by the induction hypothesis \eqref{eq:induction} and Lemma~\ref{lem:diagonal}, we find that (suppressing the dependence on $n\geq2$ for readability) 
	\begin{align*}
		(J_\alpha-\lambda) y_{k+1} &=(J_\alpha-\lambda)y_k-\frac{\alpha C_{\alpha,s,k}(k)}{2 k(s+k)} (J_\alpha-\lambda)\zeta_{\frac{1}{2}+s+2k}\\
		&=\sum_{j=k}^\infty C_{\alpha,s,k}(j) \zeta_{\frac{1}{2}+s+2j}-\frac{\alpha C_{\alpha,s,k}(k)}{2 k(s+k)} \left(\frac{\alpha^2-\beta^2}{2\alpha}-\lambda\right)\zeta_{\frac{1}{2}+s+2k}\\
		&\qquad-\frac{\alpha C_{\alpha,s,k}(k)}{2 k(s+k)}\sum_{j=1}^\infty C_{\alpha,\beta}(j)\zeta_{\frac{1}{2}+s+2k+2j} \\
		&= \sum_{j=k+1}^\infty C_{\alpha,s,k}(j) \zeta_{\frac{1}{2}+s+2j}- \frac{\alpha C_{\alpha,s,k}(k)}{2k(s+k)}\sum_{j=1}^\infty C_{\alpha,\beta}(j) \zeta_{\frac{1}{2}+s+2k+2j}, 
	\end{align*}
	where $C_{\alpha,\beta}$ is from \eqref{eq:diagonal} with $\beta=s+2k$. The right hand side here is of the form \eqref{eq:induction} as desired. We recall that
	\[Y_{\alpha,s}(k) = -\frac{\alpha C_{\alpha,s,k}(k)}{2 k(s+k)}\]
	where $C_{\alpha,s,k}(k)$ is a finite combination of $Y_{\alpha,s}(j)$ for $0\leq j\leq k-1$ and coefficients from \eqref{eq:diagonal}. The latter are polynomials in $s$, which in total demonstrates that the function $s \mapsto Y_{\alpha,s}(k)$ is rational with simple poles at $s=-1,-2,\ldots,-k$. The claims for $C_{\alpha,s,k}(j)$ follow similarly. 
\end{proof}

\subsection{Estimates for an auxiliary operator} In the proof of Lemma~\ref{lem1} below, we will need a certain auxiliary upper-triangular operator $V_\alpha$ (in the terminology of ODEs, this corresponds to a Volterra type integral operator). We need to prepare an estimate for this operator. Given $\beta > 0$, we denote by $\ell^\infty_\beta$ the Banach space of sequences $\xi$ with finite norm
\[\|\xi\|_{\ell^\infty_\beta} = \sup_{n\geq1} |\xi(n)| n^{\frac{1}{2} + \beta};\]
we have the embedding $\ell^\infty_\beta\subseteq \ell^2$ for any $\beta>0$. We also denote the first standard basis vector by $e_1=(1,0,0,\dots)\in \ell^2$. 
\begin{lemma}\label{lem:inverse} 
	Fix $0<\alpha<\infty$ and let $\beta>\alpha$. Suppose that $\xi\in\ell^\infty_\beta$, and let $\eta=V_\alpha \xi$, where 
	\begin{equation}\label{d1} 
		(V_\alpha\xi)(n)=-n^{-\frac{1}{2}-\alpha}\sum_{m=n+1}^\infty m^{-\frac{1}{2}+\alpha} \left(1-\left(\frac{n}{m}\right)^{2\alpha}\right)\xi(m). 
	\end{equation}
	Then $\eta \in \ell^\infty_\beta$ with the estimate 
	\begin{equation}\label{d5} 
		\|\eta\|_{\ell^\infty_\beta} \leq \frac{\|\xi\|_{\ell^\infty_\beta}}{\beta - \alpha}, 
	\end{equation}
	and $\eta$ satisfies the equation
	\[J_\alpha \eta=\xi- \langle \xi,\, \zeta_{\frac{1}{2}-\alpha}\rangle e_1, \]
	where $\zeta_s$ is defined in \eqref{eq:zs}. 
\end{lemma}
\begin{remark}
	Informally speaking, this lemma tells us that if we seek to solve the equation $J_\alpha\eta=\xi$, then instead of $\eta=J_\alpha^{-1}\xi$, we can take the solution $\eta=V_\alpha\xi$, if we are prepared to pay the price of adding the term $\langle \xi,\, \zeta_{\frac{1}{2}-\alpha}\rangle e_1$ to the equation (i.e. if we do not care about the initial condition \eqref{c2}). The advantage of this is the weighted norm estimate \eqref{d5}. 
\end{remark}
\begin{proof}
	The first claim follows easily by an integral estimate,
	\[|\eta(n)| \leq \|\xi\|_{\ell^\infty_\beta} n^{-\frac{1}{2}-\alpha} \int_n^\infty x^{-\frac{1}{2}+\alpha-\frac{1}{2}-\beta}\,dx = \frac{\|\xi\|_{\ell^\infty_\beta}}{\beta-\alpha} n^{-\frac{1}{2}-\beta}.\]
	Next, we note that
	\[(I-S)D(n^{\frac{1}{2}+\alpha})\zeta_{\frac{1}{2}+\alpha}=e_1,\]
	and therefore, by the factorisation \eqref{eq:factorisation},
	\[J_\alpha \zeta_{\frac{1}{2}+\alpha}=D(n^{\frac{1}{2}+\alpha})(I-S^\ast)D(c_\alpha)e_1=e_1.\]
	Now recall that $J_\alpha K_\alpha\xi=\xi$; let us rewrite $K_\alpha\xi$ as 
	\begin{align*}
		(K_\alpha \xi) (n) &= n^{-\frac{1}{2}-\alpha} \sum_{m=1}^n m^{-\frac{1}{2}+\alpha} \xi(m) + n^{-\frac{1}{2}+\alpha} \sum_{m=n+1}^\infty m^{-\frac{1}{2}-\alpha} \xi(m) \\
		&=\langle \xi,\, \zeta_{\frac{1}{2}-\alpha} \rangle n^{-\frac{1}{2}-\alpha}-n^{-\frac{1}{2}-\alpha}\sum_{m=n+1}^\infty m^{-\frac{1}{2}+\alpha}\xi(m) \\
		&\qquad\qquad\qquad\qquad\qquad\qquad\qquad+ n^{-\frac{1}{2}+\alpha} \sum_{m=n+1}^\infty m^{-\frac{1}{2}-\alpha} \xi(m) \\
		&=\langle \xi,\, \zeta_{\frac{1}{2}-\alpha} \rangle n^{-\frac{1}{2}-\alpha}-n^{-\frac{1}{2}-\alpha}\sum_{m=n+1}^\infty m^{-\frac{1}{2}+\alpha} \left(1-\left(\frac{n}{m}\right)^{2\alpha}\right)\xi(m) \\
		&=\langle \xi,\, \zeta_{\frac{1}{2}-\alpha} \rangle\zeta_{\frac{1}{2}+\alpha}(n)+V_\alpha\xi(n). 
	\end{align*}
	Finally we find that
	\[J_\alpha \eta=J_\alpha V_\alpha\xi=J_\alpha\big(K_\alpha\xi-\langle \xi,\, \zeta_{\frac{1}{2}-\alpha} \rangle\zeta_{\frac{1}{2}+\alpha}\big)=\xi-\langle \xi,\, \zeta_{\frac{1}{2}-\alpha} \rangle e_1,\]
	as claimed. 
\end{proof}

\subsection{Proof of Lemma~\ref{lem1}} Fix $0<\alpha<\infty$. Throughout the proof, set
\[\lambda = \frac{\alpha^2-s^2}{2\alpha}\]
for $\mre{s}\geq0$.

\emph{Uniqueness:} We start with the proof of uniqueness. Let $\mre{s}\geq0$. Suppose that $x_{\alpha,s}$ is a solution to the difference equation \eqref{eq:diffeq} which satisfies the asymptotic estimate 
\begin{equation}\label{d3} 
	x_{\alpha,s}(n) = n^{-\frac{1}{2}-s}+O_{\alpha,s}(n^{-\frac{1}{2}-\mre{t}-2}), 
\end{equation}
as $n\to\infty$. We then want to prove that $x_{\alpha,s}$ is the unique solution that satisfies this asymptotic estimate. There are three cases. 
\begin{enumerate}
	\item[\rm (i)] $s=it$ for $t \in \mathbb{R}\backslash\{0\}$, then $\lambda$ is real. Therefore, if $x_{\alpha,s}$ solves \eqref{eq:diffeq}, then so does its conjugate, which satisfies an asymptotic estimate with leading term $n^{-\frac{1}{2}+it}$. The fact that there can only be two linearly independent solutions of \eqref{eq:diffeq} means that $x_{\alpha,s}$ is the unique solution with the stated asymptotic estimate. 
	\item[\rm (ii)] If $\mre{s}>0$, then $x_{\alpha,s}$ is in $\ell^2$. Since $J_\alpha$ is limit point, we cannot have two linearly independent solutions to \eqref{eq:diffeq} in $\ell^2$. Hence the solution with the stated asymptotic is unique. 
	\item[\rm (iii)] When $s=0$ we refer to \cite[Equation~1.13]{WL92}, which states that there is another solution $x'_{\alpha,0}$ satisfying \eqref{eq:hereislog}. Since the space of solutions is two-dimensional, this implies the uniqueness of the solution $x_{\alpha,0}$. 
\end{enumerate}

\emph{Existence:} Given $s$ such that $\mre{s} > -1$, let $k$ be an integer such that $2k> \alpha + 1 + |\lambda|$. By Lemma~\ref{lem:induction}, we can construct a sequence $y_{k,s}$ of the form \eqref{eq:goodguy} such that the error term 
\begin{equation}\label{d2} 
	(J_\alpha-\lambda) y_{k,s}=u_{k,s} 
\end{equation}
satisfies $u_{k,s} \in \ell^\infty_{2k-1}$. We seek a solution $x = x_{\alpha,s}$ to the recurrence relation \eqref{eq:diffeq} satisfying the asymptotics \eqref{d3}. First we explain the intuition behind our argument below. We use the standard technique; in the terminology of ODEs, we reduce a solution to the ODE to the solution of a Volterra type integral equation. To ``derive'' this equation, we write
\[(J_\alpha-\lambda)x_{\alpha,s}=0\]
and subtract this from \eqref{d2}; we get
\[(J_\alpha-\lambda)(x_{\alpha,s}-y_{k,s})=-u_{k,s}.\]
Applying $J_\alpha^{-1}$, we obtain
\[x_{\alpha,s}-y_{k,s}=\lambda J_\alpha^{-1}(x_{\alpha,s}-y_{k,s})-J_\alpha^{-1}u_{k,s}.\]
Lemma~\ref{lem:inverse} tells us that we may replace $J_\alpha^{-1}$ by $V_\alpha$ here if we pay the price of adding a term proportional to $e_1$ to the equation for $x_{\alpha,s}$. 

Having explained this, we start from the equation 
\begin{equation}\label{d4} 
	x_{\alpha,s}-y_{k,s}=\lambda V_\alpha(x_{\alpha,s}-y_{k,s})-V_\alpha u_{k,s}, 
\end{equation}
where $V_\alpha$ is as in \eqref{d1}. We will consider this as an equation in $\ell^\infty_{2k-1}$. Note that by Lemma~\ref{lem:inverse}, we have $V_\alpha u_{k,s} \in \ell^\infty_{2k-1}$. Since $2k - 1 - \alpha > |\lambda|$, estimate \eqref{d5} of the same lemma shows that the operator $I - \lambda V_\alpha$ is invertible in $\ell^\infty_{2k-1}$ and therefore equation \eqref{d4} can be solved:
\[x_{\alpha,s}-y_{k,s} = -(I - \lambda V_\alpha)^{-1} V_\alpha u_{k,s}.\]
By Lemma~\ref{lem:inverse}, from \eqref{d4} we get that
\[J_\alpha(x_{\alpha,s}-y_{k,s})=\lambda(x_{\alpha,s}-y_{k,s})-u_{k,s}+Be_1\]
for some constant $B$. In other words, by comparison with \eqref{d2}, we find that $x_{\alpha,s}$ is a solution to the recurrence relation \eqref{eq:diffeq}. Recalling the expression for $y_{k,s}$ from \eqref{eq:goodguy} and, if need be, making $k$ so large that $2k -1 > \mre{s} + 2$, we see that
\[x_{\alpha,s}(n) = n^{-\frac{1}{2}-s} + O_{\alpha,s}(n^{-\frac{1}{2}-\mre{s}-2}),\]
so our solution $x_{\alpha,s}$ has the required asymptotics. 

\emph{Analyticity:} Let us prove the analyticity of $x_{\alpha,s}(n)$ in $s$ for $\mre{s}>-1$. We start from \eqref{d2} and note that by Lemma~\ref{lem:induction}, the map
\[s \mapsto u_{k,s} \in \ell^\infty_{2k-1}\]
is analytic. Furthermore, the map
\[s \mapsto x_{\alpha,s}-y_{k,s} = -(I - \lambda V_\alpha)^{-1} V_\alpha u_{k,s} \in \ell^\infty_{2k-1},\]
is analytic (for $s$ such that $|\lambda| < 2k - 1 - \alpha$). In the first part of the proof we saw that $x_{\alpha,s}$ is unique for $\mre{s} \geq 0$. Combined with the analyticity of $s \mapsto y_{k,s}(n)$ we thus find that $s \mapsto x_{\alpha,s}(n)$ extends to an analytic function in $\mre{s} > -1$, for every $n\geq1$. \qed

\section{Eigenvalues of $K_\alpha$} \label{sec:eigenvalues}

\subsection{Preliminaries} Let $0<\alpha<\infty$. By Theorem~A, the operator $K_\alpha \colon \ell^2 \to \ell^2$ has a finite number of eigenvalues, all strictly greater than $2/\alpha$. The purpose of this section is to analyse the behaviour of these eigenvalues and to prove Theorem~B.

We begin with some heuristics. Observe that
\[\lim_{\alpha\to \infty} k_\alpha(n,m) = \frac{1}{\sqrt{nm}} 
\begin{cases}
	1, & \text{if } n=m ,\\
	0, & \text{if } n\neq m. 
\end{cases}
\]
Thus, we have the (weak) convergence $K_\alpha\to K_\infty$ as $\alpha\to\infty$, where $K_\infty$ is the diagonal matrix with the elements $1,\tfrac{1}{2},\frac{1}{3},\ldots$ on the diagonal. 

Obviously, the eigenvalues of $K_\infty$ are $(j^{-1})_{j\geq1}$. Heuristically, one can think that as $\alpha$ increases, the essential spectrum of $K_\alpha$ shrinks, ``revealing'' the sequence of eigenvalues which converges (as $\alpha\to\infty$) to $(j^{-1})_{j\geq1}$. Theorem~B gives these ideas a more precise meaning. 

The following result, which also demonstrates that $K_\alpha$ converges to $K_\infty$ in the operator norm, supplies the estimates needed in the proof of Theorem~B. 
\begin{lemma}\label{lma:d2} 
	For every $0<\alpha<\infty$, we have
	\[\|K_\alpha - K_\infty\|=\frac{2}{\alpha}.\]
\end{lemma}
\begin{proof}
	By a standard use of the Cauchy--Schwarz inequality (see \cite[Chapter~IX]{HLP}), 
	\begin{align*}
		\big|\langle (K_\alpha - K_\infty) x,\, x\rangle \big| &\leq \sum_{m=1}^\infty \sum_{\substack{n=1 \\n \neq m}}^\infty k_\alpha(n,m) |x(n)| |x(m)| \\
		&\leq \sum_{m=1}^\infty |x(m)|^2 \sqrt{m} \sum_{\substack{n=1 \\
		n \neq m}}^\infty \frac{k_\alpha(n,m)}{\sqrt{n}}. 
	\end{align*}
	Since the term $n=m$ is excluded, we can estimate directly with integrals to obtain 
	\begin{align*}
		\sqrt{m} \sum_{\substack{n=1 \\
		n \neq m}}^\infty \frac{k_\alpha(n,m)}{\sqrt{n}} &= m^{-\alpha} \sum_{n=1}^{m-1} n^{\alpha-1} + m^{\alpha} \sum_{n=m+1}^\infty n^{-\alpha-1} \\
		&\leq m^{-\alpha} \int_0^m x^{\alpha-1}\,dx + m^\alpha \int_m^\infty x^{\alpha-1}\,dx = \frac{2}{\alpha}. 
	\end{align*}
	Hence we find that $\|K_\alpha-K_\infty\| \leq 2/\alpha$. The converse estimate $\|K_\alpha-K_\infty\|\geq 2/\alpha$ can be obtained either from Theorem~A, noting that compact perturbations leave the essential spectrum unchanged, or directly from testing on $x(n) = n^{-\frac{1}{2}-\varepsilon}$ for $\varepsilon>0$. For the details of this computation we refer e.g. \cite[Lemma~7]{Brevig17}. 
\end{proof}

\subsection{Proof of Theorem~B} We begin with (i). Assume that $\alpha$ is such that $K_\alpha$ has at least $j$ eigenvalues. We multiply the integral representation \eqref{eq:intrep} by $\alpha$, to find that
\[\left\langle \alpha K_\alpha x,\, x \right\rangle = \int_{-\infty}^\infty \left|\sum_{n=1}^\infty x(n) n^{-\frac{1}{2}-it}\right|^2 \frac{2\alpha^2}{\alpha^2+t^2}\,\frac{dt}{2\pi}.\]
Since the map
\[\alpha \mapsto \frac{2\alpha^2}{\alpha^2+t^2}\]
is strictly increasing for every $t\neq 0$, we conclude that if $\alpha<\beta$ and $x \not\equiv 0$, then
\[\left\langle \alpha K_\alpha x,\, x \right\rangle<\left\langle \beta K_\beta x,\, x \right\rangle.\]
By the min-max principle (see e.g. \cite[Ch.~12.1]{Schmudgen12}), we conclude that the function $\alpha \mapsto \lambda_j(\alpha K_\alpha)$ is increasing. 

This also demonstrates that the function $\alpha \mapsto N(K_\alpha)$ is non-decreasing. For the second statement of (ii), let $k>\alpha/4$ be an integer. We write
\[K_\alpha=K_\infty+(K_\alpha-K_\infty)\]
and apply the min-max principle together with Lemma~\ref{lma:d2} to the top $k$ eigenvalues of $K_\alpha$. We get, for $j=1,\dots,k$, that 
\begin{equation}\label{eq:lowerest} 
	\lambda_j(K_\alpha)\geq \lambda_j(K_\infty)-\|K_\alpha-K_\infty\| = \frac{1}{j}-\frac{2}{\alpha} \geq \frac{1}{k}-\frac{2}{\alpha} > \frac{4}{\alpha}-\frac{2}{\alpha} = \frac{2}{\alpha}, 
\end{equation}
and so $N(K_\alpha)\geq k>\alpha/4$. Hence $N(K_\alpha)$ is unbounded as $\alpha\to\infty$. 

The estimate \eqref{eq:lowerest} also implies that
\[\lambda_j(K_\alpha) \geq \frac{1}{j}-\frac{2}{\alpha},\]
which is the lower bound in the asymptotic estimate of (iii). The upper bound
\[\lambda_j(K_\alpha) \leq \frac{1}{j}+\frac{2}{\alpha}\]
is similarly obtained from the estimate $\lambda_j(K_\alpha) \leq \lambda_j(K_\infty)+\|K_\alpha-K_\infty\|$. \qed

\section{The reproducing kernel thesis for composition operators on $\mathscr{H}^2$} \label{sec:rpk} 

\subsection{Preliminaries} Recall from Section~\ref{sec:a4} that $\mathscr{H}^2$ denotes the Hilbert space of Dirichlet series $f(s)=\sum_{n\geq1}x(n) n^{-s}$, with the norm
\[\|f\|_{\mathscr{H}^2}^2 = \sum_{n=1}^\infty |x(n)|^2.\]
For the basic properties of $\mathscr{H}^2$ we refer to the monograph \cite{QQ13}. The analytic functions $\varphi\colon \mathbb{C}_{1/2}\to\mathbb{C}_{1/2}$ generating bounded composition operators $\mathscr{C}_\varphi(f) = f \circ \varphi$ on $\mathscr{H}^2$, have been classified by Gordon and Hedenmalm \cite{GH99}. Their result states that the symbol $\varphi$ generates a bounded composition operator on $\mathscr{H}^2$ if and only if it belongs to the following class. 
\begin{definition}
	The Gordon--Hedenmalm class, denoted $\mathscr{G}$, consists of the functions $\varphi\colon \mathbb{C}_{1/2}\to\mathbb{C}_{1/2}$ of the form
	\[\varphi(s) = c_0 s + \sum_{n=1}^\infty c_n n^{-s} = c_0s + \varphi_0(s),\]
	where $c_0$ is a non-negative integer and the Dirichlet series $\varphi_0$ converges uniformly in $\mathbb{C}_\varepsilon$ for every $\varepsilon>0$, in addition to satisfying the following mapping properties: 
	\begin{enumerate}
		\item[(a)] If $c_0=0$, then $\varphi_0(\mathbb{C}_0)\subseteq \mathbb{C}_{1/2}$. 
		\item[(b)] If $c_0\geq1$, then either $\varphi_0\equiv 0$ or $\varphi_0(\mathbb{C}_0)\subseteq \mathbb{C}_0$. 
	\end{enumerate}
\end{definition}
In the case (b), in which $\varphi(+\infty)=+\infty$, the norm of the composition operator is always equal to $1$. In the case (a) the norm is always strictly bigger than $1$. Problem~3 of \cite{Hedenmalm04} asks how big the norm of $\mathscr{C}_\varphi$ can be if we require that $\mre{\varphi(+\infty)}-1/2=\alpha$ for some fixed $0<\alpha<\infty$. A solution would yield an analogue to the classical sharp upper bound, obtained from Littlewood's subordination principle, for the norm of a composition operator on the Hardy space of the unit disc.

By the results of \cite{Brevig17}, we know that for fixed $0<\alpha<\infty$ an optimal symbol is
\[\varphi_\alpha(s) = \frac{1}{2}+\alpha \frac{1-2^{-s}}{1+2^{-s}}.\]
Since $K_\alpha = \mathscr{C}_{\varphi_\alpha}^\ast \mathscr{C}_{\varphi_\alpha}$, as discussed in Section~\ref{sec:a4}, we see that $\|\mathscr{C}_{\varphi_\alpha}\|^2 = \|K_\alpha\|$. Recall from Theorem~A and \eqref{eq:Brevigest} that there is some $1 < \alpha_1 < 2$ such that $K_\alpha$ has no eigenvalues if $0< \alpha \leq \alpha_1$ and at least one eigenvalue if $\alpha>\alpha_1$. Consequently, we obtain the sharp upper bound
\[\|\mathscr{C}_\varphi\|^2 \leq 
\begin{cases}
	2/\alpha, & \text{if } 0 < \alpha \leq \alpha_1, \\
	\lambda_1(K_\alpha), & \text{if } \alpha_1 < \alpha < \infty, 
\end{cases}
\]
for all $\varphi \in \mathscr{G}$ such that $\mre{\varphi(+\infty)}-1/2=\alpha>0$. Hence Theorem~A provides some new information on \cite[Problem~3]{Hedenmalm04} mentioned above. In particular, this means that if $\alpha>\alpha_1$ then the norm of $K_\alpha$ is equal to the largest eigenvalue of $K_\alpha$ (in $\ell^2$).

\subsection{Eigenvectors of $K_\alpha$} In the present section, we obtain the following key lemma needed in the proof of Theorem~D. 
\begin{lemma}\label{lem:rpk} 
	Fix $0<\alpha<\infty$. There is no real number $\beta>0$ such that 
	\begin{equation}\label{eq:eigenvector} 
		x = \big(1,2^{-\frac{1}{2}-\beta},3^{-\frac{1}{2}-\beta},\ldots\big) 
	\end{equation}
	is an eigenvector of $K_\alpha$. 
\end{lemma}

Before proceeding to the proof of Lemma~\ref{lem:rpk}, we prepare an elementary preliminary estimate. 
\begin{lemma}\label{lem:EM} 
	If $\alpha,\beta>0$ and $1 \leq \alpha-\beta \leq 2$, then for all $n\in\mathbb{N}$ 
	\begin{multline*}
		n^{-(\alpha-\beta)} \sum_{m=1}^n m^{\alpha-\beta-1} + n^{(\alpha+\beta)} \sum_{m=n+1}^\infty m^{-\alpha-\beta-1} \\
		\leq \frac{1}{\alpha-\beta} + \frac{1}{\alpha+\beta}+\frac{\alpha}{6n^2} - \frac{1}{12} n^{-(\alpha-\beta)}. 
	\end{multline*}
\end{lemma}
\begin{proof}
	Applying the standard Euler--Maclaurin summation formula, by an elementary calculation one obtains (see the first estimate of \cite[Lemma~10]{Brevig17} for the details)
	\[n^{\alpha+\beta} \sum_{m=n+1}^\infty m^{-\alpha-\beta-1} \leq \frac{1}{\alpha+\beta}-\frac{1}{2n}+\frac{\alpha+\beta+1}{12n^2}.\]
	This estimate is valid for $\alpha+\beta>0$. Similarly (see the third estimate in \cite[Lemma~10]{Brevig17}), 
	\begin{multline*}
		n^{-(\alpha-\beta)}\sum_{m=1}^n m^{\alpha-\beta-1} \leq \frac{1}{\alpha-\beta} + \frac{1}{2n} + \frac{\alpha-\beta-1}{12n^2} \\- \frac{(\alpha-\beta-3)(\alpha-\beta-4)}{12(\alpha-\beta)} n^{-(\alpha-\beta)}, 
	\end{multline*}
	which is valid for $1 \leq \alpha-\beta \leq 2$. The proof is completed by combining these two estimates and noting that the function $\xi \mapsto -(\xi-3)(\xi-4)/\xi$ attains its maximum on the interval $1 \leq \xi \leq 2$ in the endpoint $\xi=2$. 
\end{proof}
\begin{proof}
	[Proof of Lemma~\ref{lem:rpk}] Assume that $K_\alpha$ has an eigenvector of the form \eqref{eq:eigenvector}. By considering the row-wise formulation of the eigenvalue equation $K_\alpha x = \lambda x$ for some $\lambda>0$, we find that
	\[n^{-\alpha-\frac{1}{2}} \sum_{m=1}^n m^{\alpha-(1+\beta)} + n^{\alpha-\frac{1}{2}}\sum_{m=n+1}^\infty m^{-(1+\alpha+\beta)} = \lambda n^{-(\frac{1}{2}+\beta)}.\]
	Multiplying both sides with $n^{\frac{1}{2}+\beta}$ yields that the identity
	\[\lambda = F(n) = n^{-(\alpha-\beta)} \sum_{m=1}^n m^{\alpha-\beta-1} + n^{-(\alpha+\beta)} \sum_{m=n+1}^\infty m^{-\alpha-\beta-1}\]
	must hold for every positive integer $n$. In particular, $F$ must be a constant sequence. Clearly, $F(1) = \zeta(1+\alpha+\beta)$. By a Riemann sum argument we also have that 
	\begin{align*}
		F(+\infty) = \lim_{n\to\infty} F(n) &= \lim_{n\to\infty} \left(\sum_{m=1}^n \left(\frac{m}{n}\right)^{\alpha-\beta-1} \frac{1}{n} + \sum_{m=n+1}^\infty \left(\frac{m}{n}\right)^{-\alpha-\beta-1}\frac{1}{n}\right) \\
		&= \int_0^1 x^{\alpha-\beta-1} \,dx + \int_1^\infty x^{-\alpha-\beta-1}\,dx = \frac{1}{\alpha-\beta}+\frac{1}{\alpha+\beta}. 
	\end{align*}
	Note here that since $F$ by assumption is a constant sequence, the first term in the limit must converge, and therefore $\alpha>\beta$. 
	
	Next we multiply both sides of the identity $F(1)=F(+\infty)$ by $\alpha+\beta>0$, to obtain the equation 
	\begin{equation}\label{eq:F1inf} 
		(\alpha+\beta)\zeta(1+\alpha+\beta) = \frac{\alpha+\beta}{\alpha-\beta}+1. 
	\end{equation}
	The left hand side is increasing in $\alpha$, and since $\alpha>\beta$ the right hand side is decreasing in $\alpha$. Applying the estimates $1<\zeta(\sigma)<\sigma/(\sigma-1)$, valid for $1<\sigma<\infty$, we find that
	\[\alpha+\beta < (\alpha+\beta)\zeta(1+\alpha+\beta) < 1+\alpha+\beta.\]
	Inserting these estimates into \eqref{eq:F1inf} and solving the corresponding equations, we conclude that $1+\beta < \alpha < 1+ \sqrt{1+\beta^2}<2+\beta$ and hence $1 < \alpha-\beta < 2$. Now we get from Lemma~\ref{lem:EM} that
	\[\liminf_{n\to\infty} n^{\alpha-\beta} \left(F(+\infty)-F(n)\right) \geq \frac{1}{12},\]
	which is certainly impossible if $F$ is a constant sequence. Hence our assumption that $K_\alpha$ has an eigenvector of the form \eqref{eq:eigenvector} is contradicted. 
\end{proof}
\begin{remark}
	It is also possible to prove Lemma~\ref{lem:rpk} from the Jacobi matrix point of view. Indeed, since $K_\alpha$ and $J_\alpha$ have the same eigenvectors, we could equivalently have considered the eigenvalue equation $J_\alpha x = \lambda^{-1} x$ for an eigenvector $x$ of the form \eqref{eq:eigenvector}. Comparing with \eqref{eq:c1} in the proof of Lemma~\ref{lem:diagonal}, we find
	\[f_{\alpha,\beta}(z)=1/\lambda\]
	for $z=n^{-1}$ with $n=1,2,\ldots$. The left hand side is analytic for $|z|<1$, so we find that $f_{\alpha,\beta}$ is constant in the unit disk. It is not difficult to see that the last statement is false, which gives a contradiction. 
\end{remark}

\subsection{Proof of Theorem~D} The reproducing kernel of $\mathscr{H}^2$ at the point $w$ is
\[\mathscr{K}_w(s)=\zeta(s+\overline{w}) = \sum_{n=1}^\infty n^{-s-\overline{w}}.\]
For $\varphi\in\mathscr{G}$, define 
\begin{equation}\label{eq:Sphi} 
	S_\varphi = \sup_{w \in \mathbb{C}_{1/2}}\frac{\|\mathscr{C}_\varphi \mathscr{K}_w\|_{\mathscr{H}^2}}{\|\mathscr{K}_w\|_{\mathscr{H}^2}}. 
\end{equation}
The reproducing kernel thesis is the statement $S_\varphi = \|\mathscr{C}_\varphi\|$. 

Our first goal is to show that for the symbol $\varphi_\alpha$, it is sufficient to consider only real $w$ in \eqref{eq:Sphi}. To achieve this, we prove that the inequality 
\begin{equation}\label{eq:wreal} 
	\frac{\|\mathscr{C}_{\varphi_\alpha} \mathscr{K}_w\|_{\mathscr{H}^2}}{\|\mathscr{K}_w\|_{\mathscr{H}^2}} \leq \frac{\|\mathscr{C}_{\varphi_\alpha} \mathscr{K}_{\mre{w}}\|_{\mathscr{H}^2}}{\|\mathscr{K}_{\mre{w}}\|_{\mathscr{H}^2}} 
\end{equation}
holds for every $w \in \mathbb{C}_{1/2}$. Since $\|\mathscr{K}_w\|_{\mathscr{H}^2}^2 = \zeta(2\mre{w})$ the denominators are equal. Recalling that $K_\alpha = \mathscr{C}_{\varphi_\alpha}^\ast \mathscr{C}_{\varphi_\alpha}$, the estimate \eqref{eq:wreal} follows at once from the triangle inequality and the fact that $k_\alpha(n,m)\geq0$, since
\[\|\mathscr{C}_{\varphi_\alpha} \mathscr{K}_w \|_{\mathscr{H}^2}^2 = \sum_{m=1}^\infty \sum_{n=1}^\infty \frac{k_\alpha(n,m)}{n^{w} m^{\overline{w}}} \leq \sum_{m=1}^\infty \sum_{n=1}^\infty \frac{k_\alpha(n,m)}{(nm)^{\mre{w}}} = \|\mathscr{C}_{\varphi_\alpha} \mathscr{K}_{\mre{w}} \|_{\mathscr{H}^2}^2.\]
Setting $w = 1/2+\beta$ for $0<\beta<\infty$, we therefore have that 
\begin{equation}\label{eq:realsup} 
	S_{\varphi_\alpha} = \sup_{0<\beta<\infty} \frac{\|\mathscr{C}_{\varphi_\alpha} \mathscr{K}_{1/2+\beta}\|_{\mathscr{H}^2}}{\|\mathscr{K}_{1/2+\beta}\|_{\mathscr{H}^2}}. 
\end{equation}
We now consider the endpoints in \eqref{eq:realsup}. Is is easy (consult the proof of \cite[Lemma~7]{Brevig17} or \cite[Chapter~IX]{HLP} for the first limit) to verify that
\[\lim_{\beta\to 0^+}\frac{\|\mathscr{C}_{\varphi_\alpha} \mathscr{K}_{1/2+\beta}\|_{\mathscr{H}^2}}{\|\mathscr{K}_{1/2+\beta}\|_{\mathscr{H}^2}} = \sqrt{\frac{2}{\alpha}} \qquad \text{and} \qquad \lim_{\beta\to\infty}\frac{\|\mathscr{C}_{\varphi_\alpha} \mathscr{K}_{1/2+\beta}\|_{\mathscr{H}^2}}{\|\mathscr{K}_{1/2+\beta}\|_{\mathscr{H}^2}} =1.\]
Note that the right endpoint is of no relevance, since $\|\mathscr{C}_{\varphi_\alpha}\|^2 = \|K_\alpha\|>1$ holds for every $0<\alpha<\infty$. Hence, there are two cases to consider.

\emph{Case 1:} $\|\mathscr{C}_{\varphi_\alpha}\|^2 = 2/\alpha$. In this case we see from the left endpoint in \eqref{eq:realsup} that $S_{\varphi_\alpha} = \|\mathscr{C}_{\varphi_\alpha}\|$ and so the reproducing kernel thesis holds. Since $\|\mathscr{C}_{\varphi_\alpha}\|^2 = \|K_\alpha\|$, we get from Theorem~A that $K_\alpha$ has no eigenvalues (since they all lie above $\tfrac{2}{\alpha}$), and thus $N(K_\alpha)=0$. 

\emph{Case 2:} $\|\mathscr{C}_{\varphi_\alpha}\|^2 > 2/\alpha$. Since $\|\mathscr{C}_{\varphi_\alpha}\|^2 = \|K_\alpha\|$ we see from Theorem~A that the norm of $K_\alpha$ is equal to its largest eigenvalue, and in particular $N(K_\alpha)\geq1$. 

Let us assume that $S_{\varphi_\alpha} = \|\mathscr{C}_{\varphi_\alpha}\|$. By the computation above, this means that the norm of $\mathscr{C}_{\varphi_\alpha}$ must be attained at the reproducing kernel $\mathscr{K}_{1/2+\beta}$ for some $0<\beta<\infty$. Hence we have
\[\lambda_1(K_\alpha) = \|K_\alpha\| = \|\mathscr{C}_{\varphi_\alpha}\|^2 = S_{\varphi_\alpha}^2 = \frac{\|\mathscr{C}_{\varphi_\alpha} \mathscr{K}_{1/2+\beta}\|_{\mathscr{H}^2}^2}{\|\mathscr{K}_{1/2+\beta}\|_{\mathscr{H}^2}^2} = \frac{\langle K_\alpha x,\, x \rangle}{\|x\|_{\ell^2}^2}\]
for $x = \big(1,2^{-\frac{1}{2}-\beta},3^{-\frac{1}{2}-\beta},\ldots\big)$. However, this is impossible by Lemma~\ref{lem:rpk}, and therefore $S_{\varphi_\alpha} < \|\mathscr{C}_{\varphi_\alpha}\|$ so the reproducing kernel thesis does not hold. \qed

\bibliographystyle{amsplain} 
\bibliography{spectrum}

\providecommand{\bysame}{\leavevmode\hbox to3em{\hrulefill}\thinspace}
\providecommand{\MR}{\relax\ifhmode\unskip\space\fi MR }
\providecommand{\MRhref}[2]{%
  \href{http://www.ams.org/mathscinet-getitem?mr=#1}{#2}
}
\providecommand{\href}[2]{#2}
\begin{thebibliography}{10}

\bibitem{ABT96}
Matthew~J. Appel, Paul~S. Bourdon, and John~J. Thrall, \emph{Norms of
  composition operators on the {H}ardy space}, Experiment. Math. \textbf{5}
  (1996), no.~2, 111--117.

\bibitem{Brevig17}
Ole~Fredrik Brevig, \emph{Sharp norm estimates for composition operators and
  {H}ilbert-type inequalities}, Bull. Lond. Math. Soc. \textbf{49} (2017),
  no.~6, 965--978.

\bibitem{BPSSV}
Ole~Fredrik Brevig, Karl-Mikael Perfekt, Kristian Seip, Aristomenis~G.
  Siskakis, and Dragan Vukoti\'{c}, \emph{The multiplicative {H}ilbert matrix},
  Adv. Math. \textbf{302} (2016), 410--432.

\bibitem{GP}
D.~J. Gilbert and D.~B. Pearson, \emph{On subordinacy and analysis of the
  spectrum of one-dimensional {S}chr\"{o}dinger operators}, J. Math. Anal.
  Appl. \textbf{128} (1987), no.~1, 30--56.

\bibitem{GH99}
Julia Gordon and H{\aa}kan Hedenmalm, \emph{The composition operators on the
  space of {D}irichlet series with square summable coefficients}, Michigan
  Math. J. \textbf{46} (1999), no.~2, 313--329.

\bibitem{HLP}
G.~H. Hardy, J.~E. Littlewood, and G.~P\'{o}lya, \emph{Inequalities}, Cambridge
  Mathematical Library, Cambridge University Press, Cambridge, 1988, Reprint of
  the 1952 edition.

\bibitem{Hedenmalm04}
H{\aa}kan Hedenmalm, \emph{Dirichlet series and functional analysis}, The
  legacy of {N}iels {H}enrik {A}bel, Springer, Berlin, 2004, pp.~673--684.

\bibitem{KS}
Tom\'{a}\v{s} Kalvoda and Pavel \v{S}\v{t}ov\'{\i}\v{c}ek, \emph{A family of
  explicitly diagonalizable weighted {H}ankel matrices generalizing the
  {H}ilbert matrix}, Linear Multilinear Algebra \textbf{64} (2016), no.~5,
  870--884.

\bibitem{KP92}
S.~Khan and D.~B. Pearson, \emph{Subordinacy and spectral theory for infinite
  matrices}, Helv. Phys. Acta \textbf{65} (1992), no.~4, 505--527.

\bibitem{KLS}
Roelof Koekoek, Peter~A. Lesky, and Ren\'{e}~F. Swarttouw, \emph{Hypergeometric
  orthogonal polynomials and their {$q$}-analogues}, Springer Monographs in
  Mathematics, Springer-Verlag, Berlin, 2010, With a foreword by Tom H.
  Koornwinder.

\bibitem{MP18}
Nazar Miheisi and Alexander Pushnitski, \emph{A {H}elson matrix with explicit
  eigenvalue asymptotics}, J. Funct. Anal. \textbf{275} (2018), no.~4,
  967--987.

\bibitem{MP20}
\bysame, \emph{Restriction theorems for {H}ankel operators}, Studia Math.
  \textbf{254} (2020), no.~1, 1--21.

\bibitem{MPQ18}
Perumal Muthukumar, Saminathan Ponnusamy, and Herv\'{e} Queff\'{e}lec,
  \emph{Estimate for norm of a composition operator on the {H}ardy-{D}irichlet
  space}, Integral Equations Operator Theory \textbf{90} (2018), no.~1, Art.
  11, 12.

\bibitem{Otte}
P.~Otte, \emph{Diagonalization of the {H}ilbert matrix}, Conference talk at
  ICDESFA, 2005, Available at
  \url{https://homepage.ruhr-uni-bochum.de/Peter.Otte/publications.html}.

\bibitem{PP18}
Karl-Mikael Perfekt and Alexander Pushnitski, \emph{On the spectrum of the
  multiplicative {H}ilbert matrix}, Ark. Mat. \textbf{56} (2018), no.~1,
  163--183.

\bibitem{pushnitski2021}
Alexander Pushnitski, \emph{The spectral density of {H}ardy kernel matrices},
  to appear in J. Operator Theory (arXiv:2103.12642).

\bibitem{QQ13}
Herv\'{e} Queff\'{e}lec and Martine Queff\'{e}lec, \emph{Diophantine
  approximation and {D}irichlet series}, Harish-Chandra Research Institute
  Lecture Notes, vol.~2, Hindustan Book Agency, New Delhi, 2013.

\bibitem{Rosenblum58}
Marvin Rosenblum, \emph{On the {H}ilbert matrix. {II}}, Proc. Amer. Math. Soc.
  \textbf{9} (1958), 581--585.

\bibitem{Schmudgen12}
Konrad Schm\"{u}dgen, \emph{Unbounded self-adjoint operators on {H}ilbert
  space}, Graduate Texts in Mathematics, vol. 265, Springer, Dordrecht, 2012.

\bibitem{Schur}
J.~Schur, \emph{Bemerkungen zur {T}heorie der beschr\"{a}nkten {B}ilinearformen
  mit unendlich vielen {V}er\"{a}nderlichen}, J. Reine Angew. Math.
  \textbf{140} (1911), 1--28.

\bibitem{Teschl}
Gerald Teschl, \emph{Jacobi operators and completely integrable nonlinear
  lattices}, Mathematical Surveys and Monographs, vol.~72, American
  Mathematical Society, Providence, RI, 2000.

\bibitem{Stampach22}
Franti\v{s}ek \v{S}tampach, \emph{The {H}ilbert {$L$}-matrix}, J. Funct. Anal.
  \textbf{282} (2022), no.~8, Paper No. 109401, 46.

\bibitem{Wilf}
Herbert~S. Wilf, \emph{On {D}irichlet series and {T}oeplitz forms}, J. Math.
  Anal. Appl. \textbf{8} (1964), 45--51.

\bibitem{WL92}
R.~Wong and H.~Li, \emph{Asymptotic expansions for second-order linear
  difference equations}, J. Comput. Appl. Math. \textbf{41} (1992), no.~1-2,
  65--94, Asymptotic methods in analysis and combinatorics.

\end{thebibliography}

\end{document}